\documentclass[11pt]{amsart}

\usepackage{amsmath,amssymb,amscd,amsfonts,amsthm,verbatim}
\usepackage[all,cmtip]{xy}
\usepackage{paralist}
\usepackage[mathscr]{eucal}
\usepackage{color}
\usepackage{pdfsync}
\usepackage[]{fontenc}
\usepackage{enumerate}
\usepackage{pgfplots}
\pgfplotsset{width=7cm,compat=1.3}
\usepackage[pdftex]{hyperref}
\hypersetup{citecolor=blue,linktocpage}

\usepackage[colorinlistoftodos]{todonotes}
\newcommand\lsi[1]{\todo[inline,color=pink!60]{#1}} 
\newcommand\ma[1]{\todo[color=cyan!50]{#1}}
\newcommand\mai[1]{\todo[inline,color=cyan!50]{#1}}
\newcommand\rp[1]{\todo[color=yellow]{#1}}
\newcommand\rpi[1]{\todo[inline,color=yellow]{#1}}

\newtheorem{thm}{Theorem}[section]
\newtheorem{lem}[thm]{Lemma}
\newtheorem{prop}[thm]{Proposition}
\newtheorem{hyp}[thm]{Hypotheses}
\newtheorem{cor}[thm]{Corollary}

\newtheorem{assu-nota}[thm]{Assumption--Notation}
\newtheorem{nota}[thm]{Notation}
\newtheorem*{thm*}{Theorem}

\theoremstyle{definition}
\newtheorem{defn}[thm]{Definition}
\newtheorem{rem}[thm]{Remark}
\newtheorem{ex}[thm]{Example}
\newtheorem{qst}[thm]{Question}

\newcommand{\inv}{^{-1}}
\newcommand{\unu}{^{\nu}}

\newcommand{\into}{\hookrightarrow}

\newcommand{\C}{\mathbb C}
\newcommand{\Z}{\mathbb Z}
\newcommand{\Q}{\mathbb Q}
\newcommand{\R}{\mathbb R}

\newcommand{\N}{\mathbb N}
\newcommand{\pp}{\mathbb P}

\newcommand{\ud}{^{(d)}}
\newcommand{\up}[1]{^{(#1)}}
\newcommand{\upq}[1]{^{[#1]}}

\newcommand{\cC}{\mathcal C}

\newcommand{\OO}{\mathcal O}
\newcommand{\cO}{\mathcal O}

\newcommand{\cV}{\mathcal V}
\newcommand{\cZ}{\mathcal Z}

\DeclareMathOperator{\Aut}{Aut}
\DeclareMathOperator{\Pic}{Pic}
\DeclareMathOperator{\pic}{Pic}

\DeclareMathOperator{\Alb}{Alb}
\DeclareMathOperator{\alb}{alb}

\DeclareMathOperator{\Proj}{Proj}

\DeclareMathOperator{\rk}{rk}

\DeclareMathOperator{\pr}{pr}

\DeclareMathOperator{\kod}{kod}
\DeclareMathOperator{\vol}{vol}
\DeclareMathOperator{\sym}{Sym}

\DeclareMathOperator{\Gal}{Gal}

\DeclareMathOperator{\sing}{Sing}

\newcommand{\fie}{\varphi}

\newcommand{\cP}{\mathcal{P}}

\newcommand{\wt}{\widetilde}

\numberwithin{equation}{section}

\title{Linear systems on irregular varieties}
\author{Miguel \'Angel Barja}
\author{Rita Pardini}
\author{Lidia Stoppino}
\address{Miguel \'Angel Barja\\Departament de Matem\`atiques\\Universitat Polit\`ecnica de Catalunya\\Avda. Diagonal 647\\08028 Barcelona\\Spain}
\email{miguel.angel.barja@upc.edu}
\address{Rita Pardini\\Dipartimento di Matematica\\Universit\`a di Pisa \\Largo B. Pontecorvo 5\\I-56127  Pisa\\Italy}
\email{rita.pardini@unipi.it}
\address{Lidia Stoppino \\Dipartimento di Scienza e Alta Tecnologia\\Universit\`a dell'Insubria \\ Via Valleggio 11\\22100, Como\\Italy}
\email{lidia.stoppino@uninsubria.it}

\thanks{The first author was supported by MINECO MTM2015-69135-P ``Geometr{\'\i}a y Topolog{\'\i}a de Variedades, \'Algebra y Aplicaciones" and by Generalitat de Catalunya SGR2014-634. The second and third authors are members of G.N.S.A.G.A.--I.N.d.A.M.  This research was partially supported  and by MIUR (Italy) through  PRIN 2010-11 ``Geometria delle variet\`a algebriche" and PRIN 2012-13 ``Moduli, strutture geometriche e loro applicazioni''.}

\begin{document}
\begin{abstract}
Let $X$ be a  normal  complex projective variety,  $T\subseteq  X$  a subvariety,  $a\colon X\rightarrow A$ a morphism  to an abelian variety such that $\pic^0(A)$ injects into $\pic^0(T)$  and let
$L$ be a line bundle on $X$.

 Denote by $X\up{d}\to X$ the connected \'etale cover induced by the $d$-th multiplication map of $A$, by $T\ud \subseteq X\ud$ the preimage of $T$ and by $L\up{d}$ the pull-back of $L$ to $X\up{d}$.
For    $\alpha\in \Pic^0(A)$ general, we study the restricted  linear system $|L\up{d}\otimes a^*\alpha|_{|T\ud}$:  if  for some $d$ this  gives a generically finite  map $\varphi\up{d}$, we  show that  $\varphi\up{d}$ is  independent of $\alpha$ or $d$ sufficiently large and divisible, and  is induced by  the {\em eventual  map} $\varphi\colon T\to Z$ such that $a_{|T}$ factorizes through $\varphi$.

The generic value  $h^0_a(X_{|T}, L)$ of $h^0(X_{|T}, L\otimes\alpha)$ is called the {\em (restricted) continuous rank}. We prove that if $M$ is the pull back of an ample divisor of $A$,  then $x\mapsto h^0_a(X_{|T}, L+xM)$ extends  to  a continuous function  of $x\in\R$, which is differentiable except possibly at countably many points; when $X=T$ we compute the left derivative explicitly.

In the case when $X$ and $T$ are smooth, combining the above results  we prove  {\em Clifford-Severi type inequalities}, i.e.,  geographical bounds  of the form $$\vol_{X|T}(L)\ge C(m) h^0_a(X_{|T},L),$$ where $C(m)={\mathcal O}(m!)$.

\end{abstract}
\maketitle
\setcounter{tocdepth}{1}
\tableofcontents
\section{Introduction}
The aim of this paper is to explore a new notion of asymptotic behaviour of line bundles on irregular varieties, that  we call {\em eventual} behaviour, and prove some fundamental results.
These  results  have striking formal analogies with the usual asymptotic study, as in \cite{lazarsfeld-I} and \cite{lazarsfeld-II}, but  are in fact quite different, and so are the proofs. We introduce the {\em eventual map},  a new way of associating a map to a line bundle on a variety of maximal Albanese dimension, only formally remindful of the Iitaka fibration, and the {\em continuous rank} function,  a continuous function defined on a line in the  space of $\R$-divisor classes that has properties  similar to those of the volume function.
In the last part of the paper  we make the  relation between continuous rank and volume precise by giving   explicit lower bounds for their  ratio   that imply, as a special case, new strong geographical bounds for irregular  varieties of general type.
\smallskip

We work in the following relative setup.
Let $X$ be a   normal complex projective variety, let $T\subseteq X$ be a   subvariety and let  $a\colon X\rightarrow A$ be a morphism  to an abelian variety.  Assume that $a_{|T}$  is {\em strongly generating}, i.e.  that the induced homomorphism $\pic^0(A)\rightarrow \pic^0(T)$ is injective. Notice that this condition implies in particular  that $\pic^0(A)\to \pic^0(X)$ is also injective;  so we identify $\pic^0(A)$  with its image in $\pic^0(X)$ and for $\alpha\in \pic^0(A)$ we denote $a^*\alpha$ simply by $\alpha$.

 For  any integer  $d\geq 1$ consider the connected variety $X\up{d}$ defined by the following cartesian diagram, where   $\mu_d$ is  multiplication by  $d$  on $A$:
\begin{equation}\label{diag:-1}
\begin{CD}
X\up{d}@>\wt{\mu_d}>>X\\
@V{a_d}VV @VV a V\\
A@>{\mu_d}>>A
\end{CD}
\end{equation}
and set $T\ud:=\wt\mu_d^*(T)$. 
We fix   $L\in \pic(X)$,  set  $L\up{d}=\wt\mu_d^*(L)$ and
we study the linear system $|L\up{d}\otimes \alpha|_{|T\up{d}}$ for  $\alpha\in \pic^0(A)$ general and $d$ sufficiently large and divisible.

All the results in the paper are developed in this relative setting, but  for simplicity in this introduction  we describe only   the case   when  $T =X$  and  $X$ has \underline{maximal $a$-dimension}, i.e., when $a$ is generically finite onto its image.
However, we wish to stress that  the relative setup,
considered here for the first time,
not only is intrinsically interesting but it is indispensable for the applications in the second part of the paper, since even to prove the statements in the absolute case  $X=T$ one has to use the relative version,  taking as  $T$ a general element of a suitable linear system.
\smallskip

The first part of the paper is concerned with the study of the map given by $|L\ud\otimes \alpha|$ for $d\gg 0$. 
The generic value $h^0_a(X,L)$ of $h^0(X,L\otimes \alpha)$ for $\alpha \in \Pic^0(X)$, called the {\em continuous rank}  (\cite{barja-severi}),  is a measure of positivity of $L$: indeed if $h^0_a(X,L)>0$ then $L$ is big. Our main result here (Theorem \ref{thm:factorization}) is the  existence,  when $h^0_a(X,L)>0$,  of  a generically finite dominant  rational map, the {\em eventual map}, $\fie\colon X\to Z$ such that:
\begin{itemize}
\item[ (1)] $a$ is composed with $\fie$,
\item[(2)] for $d$ large and divisible enough and $\alpha\in \Pic^0(A)$ general the map given by $|L\up{d}\otimes \alpha|$ is obtained from $\fie$ by base change with the $d$-th multiplication map.
\end{itemize}
Properties (1) and (2) characterize the eventual map up to birational isomorphism. This is a completely new way of associating a map to  a line bundle on an irregular variety via an asymptotic construction, in a situation where the Iitaka fibration is birational and therefore gives no information.
Notice also that  when $L=K_X$ and $a$ is the Albanese  map $\fie$ is a new  intrinsic invariant of varieties of maximal Albanese dimension, the {\em eventual paracanonical map}: we study this case in detail in \cite{BPS2}. Recently Jiang in \cite{Jiang} obtains further results on the characterization of this factorization and completely classifies the structure of the eventual paracanonical map in dimensions 2 and 3.

The second theme of the paper is the study of the continuous rank  $h^0_a(X,L)$ of a line bundle $L$ on $X$. One of the motivations for studying this invariant rather than $h^0(X,L)$ is its  behaviour under multiplication maps  as in \eqref{diag:-1}:
 one has $$h^0_{a_d}(X\ud,L\ud)=d^{2q}h^0_a(X,L).$$
  Using this property, it is easy to see that, given a line bundle $M=a^*H$ with $H$ ample  on $A$, one can define in a natural way   $h^0_a(X,L+xM)$ for rational values of $x$.
  We   prove that this function extends to a continuous convex function on $\R$  and  compute its left derivative. This ``continuous continuous rank function'' is  a subtle invariant, that is not easy to compute explicitly  (see Examples \ref{ex: theta} and \ref{ex: jiang} and Question \ref{q: polynomial}),  and we believe it will have very interesting applications \footnote{Jiang and Pareschi in \cite{JP} study some natural generalizations of the continuous rank, and provide evidence towards a positive answer to Question \ref{q: polynomial}}.
Here we use it to prove several new Clifford-Severi inequalities (cf. \S \ref{sec: C-S revisited}). These are inequalities of the form:
 $$
 \vol(L)\ge C(n) h^0_a(X,L),
 $$
where $C(n)$ is an explicit positive  constant depending on the dimension  $n$ of $X$.

Inequalities of this type can be regarded as a quantitative version of the remark, made at the beginning of the introduction, that if $h^0_a(X,L)>0$ then $L$ is big, and therefore $\vol(L)>0$. The main results of \cite{barja-severi} (cf. also \cite{Zhang} for the case $L=K_X$) are the following: for any nef $L$ we have
\begin{equation}\label{eq: gen-Severi1}\vol(L)\ge n! h^0_a(X,L),
\end{equation}
\noindent while if $K_X-L$ is pseudoeffective:
\begin{equation}\label{eq: gen-Severi}\vol(L)\ge 2n! h^0_a(X,L).
\end{equation}

The reason for naming these type of inequalities after Clifford  is the continuous Clifford inequality
  $$\vol L=\deg L \ge 2h^0_a(X, L),$$
for a line bundle $L$ on a curve $X$ with $0\le \deg L\le 2g(X)-2$, which can be easily deduced from the usual Clifford Theorem using the covering trick introduced in \cite{Pa}.
Not only this is the simplest instance of an inequality of the type under consideration but  is the starting step of the inductive argument in  the proof of the generalized Clifford-Severi inequalies.
When $X$ is a minimal surface of general type and maximal Albanese dimension and $L=K_X$, the inequality \eqref{eq: gen-Severi} reads
\begin{equation}\label{eq: severi}
K^2_X\ge 4\chi(K_X),
\end{equation}
and is known as the ``Severi inequality''. It  has a long history. Severi (\cite{sev}) stated it  as a theorem
in 1932, but his proof was not correct (cf. \cite{cat}). Around  the end of the   1970's it was rediscovered and posed as
a conjecture  (\cite{miles}, \cite{cat}) that motivated intense  work on surfaces of general type,
including  the foundational paper \cite{xiao},
where  the conjecture is proven for  a surface fibered over
a curve of positive genus and
\cite{manetti},  where it is proven under the additional assumption that
the surface has ample canonical bundle. An unconditional proof was finally given by the second named author in 2005 (\cite{Pa}).

The Severi inequality can be written also as $c_1^2\ge \frac 12 c_2$, where $c_1, c_2$ are the Chern classes of the minimal surface $X$, and therefore  is an inequality between topological invariants. Probably it was this fact that hid  for some years that    the natural generalizations \eqref{eq: gen-Severi1} and  \eqref{eq: gen-Severi} of the Severi inequality to  nef  line bundles on variety of  any dimension $n$.

Here, using the properties of the continuous rank function, we give a slick proof of the  inequalities \eqref{eq: gen-Severi1} and \eqref{eq: gen-Severi} for any (not necessarily nef) $L$ and  extend them  to the relative case. In addition we give significative improvements depending on the geometry of the map $a$.
More precisely, we  prove  stronger inequalities in the following cases:
\begin{itemize}
\item[(a)]  when $a$ is  not composed with an involution
\item[(b)]  when  $a$  is   birational onto its image
\end{itemize}
The precise statements are given in \S \ref{sec: C-S revisited}.
 To give an idea of the significance of these inequalities we spell   out here their  consequences  for $X$  a minimal $\Q$-Gorenstein $n$-fold of maximal Albanese dimension  with terminal singularities.
In this case the generalized  Clifford-Severi inequality  \eqref{eq: gen-Severi} implies
\begin{equation}\label{eq: severi0}
K_X^n\geq 2\,n!\,\chi(\omega_{X}).
\end{equation}
Under assumption (a), we are able to prove here:
\begin{equation} \label{eq:  severi1}
K_X^n\geq \frac 9 4\,n!\,\chi(\omega_{X}),
\end{equation}
while under  assumption (b) we obtain the stronger inequality:
\begin{equation}\label{eq: severi2}
K_X^n\geq \frac 5 2\,n!\,\chi(\omega_{X}).
\end{equation}
Inequality \eqref{eq: severi1} had recently been proven only for  surfaces in  \cite{LZ2}; inequality \eqref{eq: severi2} is completely  new even in the surface  case.  All three inequalities show how strongly the geometry of the Albanese map affects the numerical invariants of a variety of general type.  We do not know whether  \eqref{eq: severi1} and  \eqref{eq: severi2} are sharp: in the case of surfaces it is expected (cf. Question \ref{question: (n+1)!}) that $K^2_X\ge 6\chi(\omega_X)$ when the Albanese map is birational, and   \eqref{eq: severi2} gives the first effective result in this direction.

In the last section  we work out several examples and pose some questions.

\bigskip

\noindent{\bf Acknowledgements:}
We would like to thank   Mart\'i Lahoz, Yongnam Lee, Yusuf Mustopa,  Gianluca Pacienza,  Giuseppe Pareschi, Gian Pietro Pirola and Angelo Vistoli for useful mathematical discussions. Special thanks go to Zhi Jiang, for detecting a mistake in an earlier version of this paper.
The second and third named author would like to thank the Departament de Matem\`atiques of the Universitat Polit\`ecnica de Catalunya for the invitation and the warm hospitality in November 2015 and September 2016.

\section{Set-up and preliminaries}\label{sec: preliminaries}

\subsection{Notation and conventions}\label{ssec: conventions}
We work over the complex numbers; varieties (and subvarieties) are assumed to be irreducible and projective.
\smallskip

In this paper the focus is on birational geometry, so  a   {\em map}  is a   rational map and we denote all maps by solid arrows.  \par
\noindent Given maps $f\colon X\to Y$ and $g\colon X\to Z$, we say that {\em $g$ is composed with $f$} if there exists a map $h\colon Y\to Z$ such that $g=h\circ f$. Given a map  $f\colon X\to Y$  and an involution $\sigma$ of $X$, we say that {\em $f$ is composed with the involution $\sigma$} if $f\circ\sigma=f$.
\par
\noindent We say that two dominant maps $f\colon X\to Z$, $f'\colon X\to Z'$ are {\em  birationally equivalent} if there exists a birational isomorphism $h\colon Z\to Z'$ such that $f'=h\circ f$.

\noindent A map $f\colon X\to Y$ is said to be {\em generically finite} if $\dim f(X)=\dim X$, i.e., we do not require that $f$ be dominant.

\noindent If $a\colon X\to A$ is a generically finite morphism to an abelian variety, we say that $X$ {\em has maximal $a$-dimension}.
If $a\colon X\to A$ is a  morphism to an abelian variety such that  the induced homomorphism  $a^*\colon \Pic^0(A)\to \Pic(X)$ is injective, we say that $a$ is {\em strongly generating}.   \smallskip

Numerical equivalence is denoted by $\equiv$; given line bundles  $L, L'\in \Pic(X)$ we write $L\le L'$ if $L'-L$ is pseudoeffective.

Let $X$ be a variety, $T\subseteq X$ a subvariety,  and $L\in \pic(X)$; we  denote by $\vol_X(L)$ the volume of $L$ on $X$ and by $\vol_{X|T}(L)$ the restricted volume of $L$ on $T$ (cf. \cite{elmnp} for the definition of the restricted volume).

Given a variety $X$  and  a subvariety $T\subseteq  X$,   we  denote
by $H^0(X_{|T},L)$ the image of the restriction map $H^0(X,L)\to H^0(T,L_{|T})$, by $h^0(X_{|T},L)$ its dimension
and by $|L|_{|T}$  the  corresponding subspace of $|L_{|T}|$.
\smallskip

If $d$ is a non-negative integer, we write ``$d\gg 0$'' instead of ``$d$ large and divisible enough''.

\subsection{Covering trick}\label{ssec: trick}
Let  $X$ be a  variety of dimension $n$ and  let $a\colon X\to A$ be  a morphism  to an abelian variety of dimension $q$.

In  this subsection  we assume in addition   that the map  $a$ is {\em strongly generating}. Note that in particular $a(X)$ generates the abelian variety $A$.

We introduce
 the following notations and geometric constructions  that will be used throughout all the paper:
\begin{itemize}
\item $H$ is a fixed very ample divisor on $A$ and we set $M=a^*H$.
\item If $d>0$ is an integer, we denote by  $\mu_d\colon A\to A$  the $d$-th multiplication map; the following cartesian diagram defines  $X\up{d}$ and the maps $a_d$ and $\wt{\mu_d}$:
\begin{equation}\label{diag:0}
\begin{CD}
X\up{d}@>\wt{\mu_d}>>X\\
@V{a_d}VV @VV a V\\
A@>{\mu_d}>>A
\end{CD}
\end{equation}
The variety $X\ud$ is irreducible since $a$ is strongly generating.
In addition, the map $a_d$ is again  strongly generating, and is generically finite if $a$ is.  This is proved in \cite{BPS} for smooth surfaces, but the proof works  without modifications for  varieties of arbitrary dimension.
\item If $T\subseteq X$ is a subvariety such that $a_{|T}$ is strongly generating, then we denote  the preimage of $T$ in $X\up{d}$  by $T\up{d}$. Note  that $T\ud$ is again irreducible.

\item For $L\in \Pic(X)$ we write  $L\ud:=\wt{\mu_d}^*L$; we also set $M_d=a_d^*H$. Notice that by \cite[Prop 2.3.5]{LB} we have that $H\equiv \frac{1}{d^2}\mu_d^*H$ and so
$M_d\equiv \frac{1}{d^2}M^{(d)}$.
\item Given $\alpha\in\Pic^0(A)$,  we denote again by $\alpha$ its pull back to $X$, $T$,  $X\up{d}$,  etc.\dots For instance, we  write $H^0(X\up{d},L\up{d}\otimes \alpha)$ instead of $H^0(X\up{d},L\up{d}\otimes a_d^*\alpha)$. Observe also that using this convention we have   $L\up{d}\otimes \alpha^d= (L\otimes \alpha)\up{d}$.
\end{itemize}

\subsection{Continuous rank}\label{ssec: rank}

Let  $X$ be a   normal  variety of dimension $n$,   let $a\colon X\to A$ be  a morphism  to an abelian variety of dimension $q$ and let $L$ be a line bundle on $X$.
As in \cite[Def. 2.1]{barja-severi}, we define the   {\em continuous rank  of $L$ (with respect to $a$)}   as  the integer
$$
h_a^0(X,L):=\min\{h^0(X,L\otimes \alpha) \mid \alpha \in \Pic^0(A)\}.
$$

\begin{rem}
When $X$ is smooth and $L=K_X+D$ is the adjoint of a nef divisor $D$,  we have $h^0_a(X,L)=\chi(L)$ by generic vanishing  (cf. \cite[Theorem B]{pp}). If in addition $D$ is  big, then  by the  Kawamata-Viehweg vanishing theorem, $h^0_a(X,L)=\chi(L)=h^0(X,L)$. In general, by \cite{CP} we have that $h^0_a(X,L)\geq h^0(X,L)-h^1(X,L)$.
\end{rem}
Given a subvariety $T\subseteq X$, there exists a non-empty open set of $\Pic^0(A)$ where  $h^0(X_{|T},L\otimes \alpha)$ is constant. We define the {\em restricted continuous rank}  $h^0_a(X_{|T},L)$ to be  this generic value.

Note that if  $T$ is a proper subvariety of $X$, then $h^0_a(X_{|T},L)$ needs not be the minimum value of  $h^0(X_{|T},L\otimes \alpha)$.

\begin{rem}\label{rem: continuous rank algebraic}
Observe that the restricted continuous rank only depends on the  class of $L$ in $\pic^0(X)/\pic^0(A)$.
\end{rem}

\begin{rem}\label{rem: continuous rank non normale} Since we assume that   $X$ is normal, the continuous rank is invariant under birational morphisms ${\widetilde X}\rightarrow X$.  Without the normality assumption on $X$, the continuous rank may increase passing from $X$ to $\widetilde X$.
\end{rem}

\begin{rem}\label{rem: difference}
By Remark \ref{rem: continuous rank non normale}, if  we  blow  up $X$ along $T$ and normalize, both $h^0_a(X,L)$  and  $h^0_a(X_{|T},L )$ stay the same. So we may reduce    to the case where $T$ is a Cartier divisor and,   in particular, we have that $h^0_a(X_{|T}, L)=h^0_a(X,L)-h^0_a(X, L-T)$ is the difference of two (non-restricted) continuous rank functions.
\end{rem}

The fundamental property of the continuous rank is the following (notation as in \S \ref{ssec: trick}):
\begin{prop}[Multiplicativity of the continuous rank]\label{prop: mult-rank}

In the above set-up, assume that $a_{|T}$ is strongly generating. Then for every $d>0$   one has:
 $$ h^0_{a_d}({X\up{d}}_{|T\ud}, L\ud)=d^{2q}h^0_a(X_{|T},L).
 $$
\end{prop}
\begin{proof}
When $T=X$, the claim is \cite[Prop.~2.8]{barja-severi}. The general case follows from this   in view of Remark \ref{rem: difference}.
\end{proof}

\begin{rem}\label{rem: strongly} Let    $a\colon X\to A$ be a morphism and  let $T\subseteq X$ be a subvariety such that
the natural maps $a^*\colon \pic^0(A)\to \pic^0(X)$ and  $a_{|T}^*\colon \pic^0(A)\to \pic^0(T)$ have the same kernel (for instance, this holds if the restriction map $\pic^0(X)\to \pic^0(T)$ is injective). Denote by $A'$   the abelian variety dual to $\pic^0(A)/\ker (a^*)$: the map  $a$ factorizes as $X\overset{a'}{\to}A'\overset{f}{\to} A$, where ${a'}_{|T}$ is a strongly generating  and $f$ is a morphism of abelian varieties. Since  one has $h^0_a(X_{|T},L)=h^0_{a'}(X_{|T},L)$, when proving  statements about  the continuous rank one can  assume that $a_{|T}$ is strongly generating and use the machinery of \S \ref{ssec: trick}.
\end{rem}

\subsection{Continuous resolution of the base locus}\label{ssec: continuous-res}
Here we  recall,   and slightly refine,   the results of  \cite[\S~3]{barja-severi}.  \smallskip

Let  $X$ be a   normal   variety of dimension $n$,   let $a\colon X\to A$ be  a morphism  to an abelian variety of dimension $q$ and let $L$ be a line bundle on $X$ with $h^0_a(X, L)>0$.  We  denote by $U_{\rk}\subseteq  \Pic^0(A)$ the open set consisting of the $\alpha$'s such that $h^0(X, L \alpha)$ is equal to the minimum $h^0_{a}(X, L)$.
The {\em continuous evaluation map} of $L$  is defined as
\begin{equation}\label{eq: eval-map}
{ev}\colon \oplus_{\alpha\in U_{\rk}}H^0(X, L\otimes \alpha)\otimes\alpha^{-1}\to  L
\end{equation}
and we denote by $S$ the subscheme  of $X$ such that the image of $ev$ is equal to $\mathcal I_SL$.

Let $\sigma\colon \widetilde X\rightarrow X$ be a smooth modification of $X$ such that   $\sigma^*\mathcal I_S=\OO_{\wt X}(-D)$ for some effective divisor $D$.
We call $D$  the {\it continuous fixed part} and $P:=\sigma^*L(-D)$ the  {\it continuous moving part} of $L$.

Assume that $a$ is strongly generating; then there exists a positive  integer $d$ such that, denoting by $g_{d}\colon {\widetilde X}\up{d\,}\rightarrow \widetilde X$  the connected \'etale cover induced by   $\mu_{d}\colon A\rightarrow A$ (see  diagram \eqref{diag:below} below),
\begin{equation}\label{diag:below}
\xymatrix{
\widetilde X\up{d\,}\ar[d]_{\tilde{a}_{d}} \ar@/^1pc/[rr]^\lambda\ar[r]_{g_{d}}&\widetilde X\ar[d]_{\tilde a}\ar[r]_\sigma &X\ar[ld]^a\\
A\ar[r]_{\mu_{d}}&A&}
\end{equation}
\
 we have
\begin{itemize}
\item for all $\alpha \in \pic^0(A)$, the system  $|P\up{d}\otimes \alpha|$ is free;
\item for  $\alpha\in \pic^0(A)$ general, $D\up{d}$ and $|P\up{d}\otimes \alpha|$ are   the fixed  and moving part, respectively,  of $|\lambda^*L\otimes \alpha|$. \end{itemize}

\begin{rem}\label{rem: continuous-restricted}
Consider now a subvariety $T\subseteq X$ such that    $T$ is not contained in $\sing X$ and $h^0_a(X_{|T},L)>0$.
It is possible to choose the modification $\sigma\colon \wt X\to X$ in such a way that $T$ is not contained in the exceptional locus of $\sigma\inv$ and therefore the strict transform $\wt T\subseteq \wt X$ is defined.
\end{rem}
\subsection{Galois group of maps}\label{ssec:aux}
We give here some   general results that we use later.
\medskip

 Given a generically finite dominant map $f\colon X\to Y$ of irreducible varieties  we denote by $\Gal(f)$ the group of birational automorphisms $\phi$ of $X$ such that $f\circ \phi=f$, namely $\Gal(f)$ is the Galois group of the field extension $\C(Y)\subseteq \C(X)$.
 \begin{prop}\label{prop: composto involuzione}
Let $f\colon X\to Y\subseteq \pp^r$ be a generically finite dominant morphism  of  varieties of dimension $\ge 2$; let $K$ be a  hyperplane section of $Y$, let $\Sigma =f^{\inv}K$ and let $h\colon \Sigma \to K$ be the restricted map.

If $K$ is general then the restriction homomorphism  $\Gal(f)\to \Gal(h)$ is an isomorphism.
\end{prop}

\begin{proof}
Possibly after removing a proper closed subset of $Y$  and its preimage in $X$, we may assume that $f\colon X\to Y$ is a finite  \'etale morphism of smooth varieties.
By \cite[Theorem~1.1~(A)]{FL}, if $K$ is general $\Sigma$ is irreducible, hence connected.

In this situation $\Gal(f)$ coincides with the group $\Gal_{top}(f)$ of deck  transformations of the topological cover $f$ and the same is true for $h$.
Choose base points $x_0\in \Sigma$ and $y_0=f(x_0) \in K$ and denote by $N_1$ (resp., $N_2$) the normalizer of $f_*\pi_1(X,x_0)$ (resp.,  $h_*\pi_1(\Sigma,x_0)$) in $\pi_1(Y,y_0)$ (resp., $\pi_1(K,y_0)$).
Then
the group $\Gal_{top}(f)$ (resp., $\Gal_{top}(h)$) can be identified with $N_1/(f_*\pi_1(X,x_0))$ (resp.,  $N_2/(h_*\pi_1(\Sigma,x_0))$). In the following commutative diagram
\[
\begin{CD}
\pi_1(\Sigma,x_0)@>>>\pi_1(X,x_0)\\
@V{h_*}VV @VV{f_*}V\\
\pi_1(K,y_0)@>>>\pi_1(Y,y_0)
\end{CD}
\]
the horizontal arrows are surjective by Theorem 1.1 (B) of \cite{FL}.
This shows that every element of $\Gal_{top}(h)$ extends to an element of $\Gal_{top}(f)$, namely the map $\Gal(f)\to \Gal(h)$ is surjective. Since $\Gal(f)\to \Gal(h)$ is  clearly  injective, this completes the proof.
\end{proof}

The next result is a straightforward generalization of \cite[Lem.~3.3]{LZ2}, but we include the proof for the reader's convenience (notation as in \S \ref{ssec: trick}).
\begin{lem}\label{lem: Gal-a}
 Let $X$ be smooth of general type and let $a\colon X\to A$ be a morphism to an abelian variety such that $a$ is strongly generating and $X$ has maximal $a$-dimension.

Then there exists a constant $C$ such that for every prime  $p>C$ and $d=p^k$, $k>0$, one has:
$$\Gal(a)=\Gal(a_d).$$
\end{lem}
\begin{proof}
Let $n:=\dim X$. By the main result of Hacon, McKernan and Xu in \cite{HMX},   there exists a constant $M$ such that for every $n$-dimensional variety  $Y$ of general type the order of the group $\Aut_{bir}(Y)$ of birational automorphisms of $Y$  is $\le M\vol_Y (K_Y)$.

Take $C=M\vol_X(K_X)$, let  $p>C$ be  a  prime and $d=p^k$ a power of $p$.
The Galois group $G\cong \Z_d^{2q}$ of $\wt {\mu_d}$ is a subgroup of $\Aut_{bir}(X\up{d})$, and it is a $p$-Sylow subgroup, because
$$|\Aut_{bir}(X\up{d})|\le M \vol_{X\up{d}}(K_{X\up{d}})=Md^{2q}\vol_X(K_X)=d^{2q}C<d^{2q}p.$$
In addition, since the number of $p$-Sylow subgroups is a divisor of  $\frac{|\Aut_{bir}(X\up{d})|}{d^2q}<p$  and is congruent to $1$ modulo $p$, it follows that  $G$ is the only $p$-Sylow subgroup, namely $G$ is a normal subgroup.

So every birational automorphism of $X\up{d}$, and thus in particular every element of $\Gal(a_d)$, descends to an automorphism of $X$. So we have a homomorphism $\Gal(a_d)\to \Gal(a)$ that is the inverse of the natural inclusion  $\Gal(a)\to \Gal(a_d)$.
\end{proof}

\subsection{Factorization of  morphisms}
In this subsection varieties are irreducible but not necessarily projective. The results  here are certainly well known but, lacking a suitable reference,  we spell them out  here for the reader's convenience.

\begin{lem}\label{lem: fGalois}
Let $f\colon X\to Y$ be  a  finite morphism of  varieties with $Y$ normal. If  $g\colon X\to Z$ is a  morphism  such that $g(f\inv(y))$ is a point for general $y\in Y$,  then $g(f\inv(y))$ is a point for all $y\in Y$.
\end{lem}
\begin{proof}
Up to replacing $f\colon X\to Y$ with its Galois closure, we may assume that $X$ is normal  and $f$ is Galois with Galois group $\Gamma$. In particular, the fibers of $f$ are $\Gamma$-orbits. Now the claim follows since for all $\gamma\in \Gamma$  the set $\{x\ |\ g(\gamma x)=g(x)\}$ is closed by definition and contains a nonempty open set by assumption, and therefore it is equal to $X$.
\end{proof}
\begin{lem}\label{lem: factors}
Let $f\colon X\to Y$ and  $g\colon X\to Z$ be   proper   morphisms of   varieties.  If   $Y$ is normal,  $f$ is surjective and   $g(f\inv(y))$ is a point for all $y\in Y$, then $g$ descends to a morphism $\bar g\colon Y\to Z$.
\end{lem}
\begin{proof}
Denote by $T\subseteq Y\times Z$ the image of the product morphism $f\times g$, which is a closed subset. The first projection $Y\times Z\to Y$ restricts to a proper bijective morphism $\pi\colon T\to Y$. Since $Y$ is normal, $\pi$ is an isomorphism and  $\bar g$ is obtained by composing $\pi\inv$ with the morphism $T\to Z$ induced by the second projection.
\end{proof}
\begin{cor}\label{cor: rational}
Let $f\colon X\to Y$ and  $g\colon X\to Z$ be   proper   morphisms of   varieties.  If     $f$ is surjective and   $g(f\inv(y))$ is a point for  $y\in Y$ general, then $g$ descends to a rational map  $\bar g\colon Y\to Z$.
\end{cor}
\begin{proof}
Follows from Lemma \ref{lem: factors} by replacing $Y$ by a smooth  open subset $Y_0$ such that $g(f\inv(y))$ for all $y\in Y_0$ and $X$ by $X_0=f\inv(Y_0)$.
\end{proof}

\section{The eventual map and the eventual degree}\label{sec: eventual}

Throughout all the section we fix:
\begin{itemize}
\item   a normal variety $X$ of dimension $n$ and  a subvariety $T\subseteq X$ of dimension $m>0$ such that $T$  is not contained in $\sing X$;
\item  a morphism $a\colon X\to A$ to an abelian variety  of dimension $q$ such that $a_{|T}$  is   strongly generating;
\item a line bundle  $L$ on $X$ with $h^0_a(X_{|T}, L)>0$.
\end{itemize}
We use freely the notation introduced in \S \ref{sec: preliminaries}.

\subsection{Eventual degree of a line bundle}\label{ssec: ev-degree}

\begin{defn}\label{def: eventual prop}
We say that  a certain property holds {\em generically}  for $L$ (with respect to $a$) iff it holds for $L \otimes \alpha$  for general $\alpha\in\Pic^0(A)$; similarly, we say that a property holds {\em eventually}   for $L$ (with respect to $a$) iff  it holds for $L\ud$ for $d\gg 0$.
\end{defn}

For instance, we say that $|L|_{|T}$ is {\em eventually generically birational } if  for $d\gg 0$  the system   $|L\ud\otimes \alpha|_{|T}$ is  birational for general $\alpha\in \Pic^0(A)$.

Using the above terminology, we formulate a partial  analogue of \cite[Thm.~3.6]{barja-severi} that we use repeatedly:
\begin{prop}\label{prop: big}
In the above setup,  if $T$ has maximal $a_{|T}$-dimension, then  $|L|_{|T}$ eventually gives a generically finite map.
In particular,  $L_{|T}$ is big.
\end{prop}
\begin{proof}
By Remark \ref{rem: continuous-restricted}, we may argue as in the proof of \cite[Thm.~3.6]{barja-severi} and  reduce to the case when $X$ is smooth and,  up to taking base change with a suitable multiplication map,   $|L\otimes \alpha |$ is a free system for every $\alpha\in \pic^0(A)$.
To prove the claim, we show that  under these assumptions if $F$ is a connected component of a general  fiber of the map  $\fie\colon T\to Z$ induced by $|L|_{|T}$,  then $a(F)$ is a point.

Indeed, the line bundle $L_{|F}$ is trivial, since $|L|$ is base point free, hence $(L\otimes\alpha)_{|F}$ and $\alpha_{|F}$  are  linearly  equivalent for every $\alpha\in \Pic^0(A)$. Since $|L\otimes \alpha|$ is also free, it follows that  $(L\otimes\alpha)_{|F}=\alpha_{|F}$ is  trivial for every $\alpha\in \Pic^0(A)$, namely the map $\pic^0(A)\to\pic^0(F)$ is trivial. So it follows that $a$ is constant on $F$.
\end{proof}

For any  given integer $d>0$, we denote by $\fie\upq{d} \colon T\up{d} \to Z\upq{d}$ the dominant map  induced by $|L\ud|_{|T\ud}$.

We start with a very useful  technical result:
\begin{lem}\label{lem:galois}
Let $d>0$ be a fixed integer and denote by $G\cong \Z_d^{2q}$ the Galois group of $\mu_d$. Then:
\begin{enumerate}
\item $G$ acts faithfully on $Z\upq{d}$;
\item if $|L\otimes \eta|_{|T}$ is base point free for every $\eta\in\Pic^0(A)[d]$, then $G$ acts freely on $Z\upq{d}$.
\end{enumerate}
\end{lem}

\begin{proof}
(i)
Since the group $G$ acts on $T\ud\subseteq X\ud$ and since the  line bundle $L\ud$, being a pull-back from $X$, has a natural $G$-linearization, there is an induced $G$-action on $Z\upq{d}$.

Let $x\in T$ be a point  that is not in the base locus of $|L\otimes \eta|$ for any $\eta\in\Pic^0(A)[d]$. We are going to show  that the points  of $W:=\wt{\mu_d}\inv(x)$ (which form a $G$-orbit of cardinality $d^{2q}$) are separated by $|L\ud|_{|T\ud}$, i.e., the natural restriction map $r\colon H^0(X\up{d}, L\ud)\to  H^0(W, {L\ud}_{|W})$ is surjective. The  map $r$ is  $G$-equivariant; the  $G$-eigenspaces of $H^0(X\up{d}, L\ud)$ are the subspaces $V_{\eta}:=\wt{\mu_d}^*H^0(X, L\otimes\eta)$ for  $\eta$  in $\Pic^0(A)[d]$, while   $H^0(W, {L\ud}_{|W})$ is isomorphic to the regular representation of $G$.  It follows that   $r$ is surjective iff for each $\eta\in \Pic^0(A)[d]$ its restriction to the  eigenspace $V_{\eta}$ is non-zero. By the choice of  $x$, for every $\eta\in \Pic^0(A)[d]$ we can find $\sigma_{\eta}\in H^0(X, L\otimes \eta)$ with $\sigma_{\eta}(x)\ne 0$; then $\tau_{\eta}:=\wt{\mu_d}^*(\sigma_{\eta})\in V_{\eta}$ does not vanish at any point of $W$, hence $r(\tau_{\eta})\ne 0$ and $r$ is surjective, as claimed.
\smallskip

(ii) In this case the argument given in (i) implies  that $|L\ud|_{|T\ud}$ is free and that all the $G$-orbits on $Z\upq{d}$ have cardinality  $d^{2q}$, and therefore $G$ acts freely.
\end{proof}

Let $d>0$ be an integer.
If   for  $\alpha\in \Pic^0(A)$ general the system $|L\ud\otimes \alpha|_{|T\ud}$  gives a generically finite map, then   we denote by $m_{L|T}(d)$ the degree of this map;  we set $m_{L|T}(d)=+\infty$ otherwise. When $T=X$ we drop $T$ from the notation and simply write $m_{L}$.

\begin{lem}\label{lem:ev-degree}
Assume that $d_1$ is an integer with $m_{L|T}(d_1)<+\infty$. If  $d_2$ is a multiple of $d_1$,  then $m_{L|T}(d_2)$ divides $m_{L|T}(d_1)$.
\end{lem}
\begin{proof}
By replacing  $(T, X ,L)$   by $(T\up{d_1},X\up{d_1}, L^{(d_1)})$, we reduce to the case where  $m_{L|T}(1)<+\infty$ and we
 show that $m_{L|T}(d)$ divides $m_{L|T}(1)$ for every $d$.

Now fix $d$;
up to replacing $L$ by $L\otimes \alpha$ for a general choice of $\alpha\in \Pic^0(A)$, we may assume  that the map $\fie$ given by $|L|_{|T}$ and the map $\fie\upq{d}$ given by $|L\ud|_{|T\ud}$ are generically finite of  degree $m_{L|T}(1)$,  $m_{L|T}(d)$, respectively; as before, we denote by $Z$, $Z\upq{d}$ the image of $\fie$, $\fie\upq{d}$ respectively.
By  the following commutative diagram
\begin{equation}\label{diag:1}
\begin{CD}
T\up{d}@>\wt{\mu_d}>>T\\
@V{\fie\upq{d}}VV @VV\fie V\\
Z\upq{d}@>h>>Z
\end{CD}
\end{equation}
we have $d^{2q}m_{L|T}(1)=m_{L|T}(d)\deg h $. By Lemma \ref{lem:galois}, the Galois group $G$ of $\mu_d$ injects into the Galois group of $h$.
It follows that $d^{2q}$ divides $\deg h$, and therefore $m_{L|T}(d)$ divides $m_{L|T}(1)$.\end{proof}

\begin{defn}\label{defn: eventual degree}
We define the {\em eventual degree of $L$ (with respect to $a$ and $T$)} as:
$$m_{L|T}:=\min\{ m_{L|T}(d)\,|\,d\in \N^* \}.$$
\end{defn}
 Note that if $T$ is of maximal $a_{|T}$ -dimension, then one has  $m_{L|T}<+\infty$ by Proposition \ref{prop: big}.

 \begin{rem}
If $m_{L|T}<+\infty$, then by Lemma   \ref{lem:ev-degree}:
\begin{itemize}
\item eventually we have $m_{L|T}(d)=m_{L|T}$;
\item $m_{L|T}$ is  the greatest common divisor of all the $m_{L|T}(d)$.
\end{itemize}
\end{rem}

\subsection{The factorization theorem}

The main result of this section is the following:

\begin{thm}[Factorization theorem] \label{thm:factorization}
Let $X$ be a normal $n$-dimensional projective variety and let $a\colon X\to A$ be a morphism to a $q$-dimensional  abelian variety; let $T\subseteq X$ be a subvariety of dimension $m$ not contained in $\sing X$ and such that $a_{|T}$ is strongly generating.

If $L$ is  a line bundle on $X$ such that $h^0_a(X_{|T}, L)>0$ and $m_{L|T}<+\infty$,
then there exists a generically finite map $\fie\colon T\to Z$ of degree $m_{L|T}$,  uniquely determined up to birational equivalence, such that:
\begin{itemize}
\item[(a)] the map $a_{|T}\colon T \to A$ is composed with $\fie$;
\item[(b)] for  $d\ge 1$ denote by $\fie\up{d}\colon T\up{d}\to Z\up{d}$ the map obtained  from  $\fie\colon T\to Z$ by   taking base change with $\mu_d$;  then
 the map given by $|L\ud\otimes \alpha|_{|T\ud}$  is composed with $\fie\up{d}$ for $\alpha\in \Pic^0(A)$ general.
\end{itemize}
In particular, for $d\gg 0$  the map $\fie\up{d}$ is birationally equivalent to the map given by $|L\ud\otimes \alpha|_{|T\ud}$ for $\alpha\in \Pic^0(A)$ general.
\end{thm}
\begin{defn}\label{def:ev-map}
We call the map $\fie\colon T\to Z$ introduced in Theorem \ref{thm:factorization} the {\em eventual map} given by $L$ on $T$. Note that  eventual degree $m_{L|T}$ of $L$ is actually the degree of the eventual map. \end{defn}

In view of Proposition \ref{prop: big},  Theorem \ref{thm:factorization} has the following immediate consequence, which will be crucial for the arguments in the second part of the paper.
 \begin{cor} \label{cor:birat} Under the assumptions of Theorem \ref{thm:factorization} we have
\begin{enumerate}
\item if $a_{|T}$ is generically injective, then the linear system $|L|_{|T}$ is eventually generically birational;
\item if $a_{|T}$ is not composed with an involution, then $m_{L|T}\ne 2$.
\end{enumerate}
\end{cor}

In order to prove Theorem \ref{thm:factorization}, we need to introduce some more notation  and  prove  a preliminary result.
Let $d>0$ be an integer; then:
\begin{itemize}
\item[--]  $U_{\rk}\ud\subseteq  \Pic^0(A)$ denotes  the nonempty open set consisting of the $\alpha$'s such that  $h^0(X\up{d}, L\ud\otimes \alpha)=h^0_{a_d}(X\up{d}, L\ud)$ and $h^0({X\up{d}}_{|T\ud}, L\ud\otimes \alpha)=h^0_{a_d}({X\up{d}}_{|T\ud}, L\ud)$;
\item[--] if $m_{L|T}(d)<+\infty$, then $U_{\deg}\ud $  denotes  the nonempty open subset of $U_{\rk}\ud $ consisting of the $\alpha$'s such that the map given by $|L\ud\otimes \alpha|_{|T}$ is generically finite of degree equal to $m_{L|T}(d)$.
\end{itemize}
We write  $U_{\rk}=U_{\rk}^{(1)}$ and $U_{\deg}=U_{\deg}^{(1)}$.
\begin{prop}\label{prop:factorization}
Assume that $|L\otimes \alpha|$ is base point free for every $\alpha\in \Pic^0(A)$ and that  $m_{L|T}(1)=m_{L|T}<+\infty$; then there exist a  variety $Z$ and  a surjective  generically finite morphism $\fie\colon T\to Z$ of degree $m_{L|T}$ such that:
 \begin{enumerate}
 \item[(a)]\label{a} the  dominant map $\fie_{\alpha}\colon T\to Z_{\alpha}$ induced by $|L\otimes \alpha|_{|T}$ is birationally equivalent  to  $\fie$ for every $\alpha\in U_{\deg};$
 \item[(b)]\label{b} the map $a_{|T}\colon T\to A$ is composed with $\fie$.
 \end{enumerate}
\end{prop}
\begin{proof}
(a)
Note first that by Lemma \ref{lem:ev-degree}, the condition $m_{L|T}(1)=m_{L|T}$ implies  $m_{L|T}(d)=m_{L|T}$ for every $d>0$.
 Up to twisting $L$  by  a very general  element of $\Pic^0(A)$,  we may assume that  $0\in U_{\deg}\ud$ for every $d\ge 1$;
we denote by $\fie\colon T\to Z$ the surjective   generically finite morphism  of degree $m_{L|T}$  given by $|L|_{|T}$. We  pick  $x\in T$ general and
consider the following continuous evaluation map on $X$:
\begin{equation}
{ev}\colon \oplus_{\alpha\in U_{\rk}}H^0(X,\mathcal I_xL\otimes \alpha)\otimes\alpha^{-1}\longrightarrow \mathcal I_x{L}
\end{equation}

The image of ${ev}$  is equal to $\mathcal I_B{L}$,  where   $B$ is a subscheme of $X$ such that
$\{x\}\subseteq  B_{|T}\subseteq  \fie^*(\fie(x))$; since $x$ is general, $\fie^*(\fie(x))=\fie\inv(\fie(x))$ is   reduced of cardinality $m_{L|T}$, hence $B_{|T}$ is also reduced and has cardinality  $\nu\le m_{L|T}$. We wish to prove that $\nu=m_{L|T}$ and therefore that $B_{|T}=\fie\inv(\fie(x))$.

Arguing  as in the proof of \cite[Lem.~3.2]{barja-severi} one can prove that $\mathcal I_B{L}$ is continuously globally generated with respect to $a$.   So by \cite[Prop.~3.1]{debarre-simple} there exists $d$ such that $\wt{\mu_d}^*(\mathcal I_B{L})\otimes\alpha$ is generated by global sections for all $\alpha\in \Pic^0(A)$. Write $B\ud:=\wt{\mu_d}^*B$, so that $\wt{\mu_d}^*(\mathcal I_B{L})=\mathcal I_{B\ud}L\ud$.

Let   $\fie\upq{d}\colon T\up{d}\to Z\upq{d}$ be the surjective morphism induced by $|L\ud|_{|T\ud}$.
By Lemma \ref{lem:galois}
the map $h$ in diagram \eqref{diag:1}  is an \'etale $G$-cover,  where   $G\cong \Z_d^{2q}$ is the Galois group of $\wt{\mu_d}$. The $0$-dimensional scheme $ {B\ud}_{|T\ud}=\wt{\mu_d}^*(B_{|T})$ is reduced of cardinality  $\nu d^{2q}$, since $B_{|T}$ is reduced of cardinality $\nu$ and $\wt{\mu_d}$ is \'etale. On the other hand, ${B\ud}_{|T\ud}$  is the base locus of $|\mathcal I_{B\ud}{L\ud}|_{|T\ud}\subseteq |L\ud|_{|T\ud}$, since  $\mathcal I_{B\ud}L\ud=\wt{\mu_d}^*(\mathcal I_BL)$ is generated by global sections.  Since  $L\ud$ is also generated by global sections, it follows that ${B\ud}_{|T\ud}$
is a union    of fibers of the map $\fie\upq{d}$.
Since in addition  ${B\ud}_{|T\ud}$ is  $G$-stable, it contains the set
$$(h\circ \fie\upq{d})\inv(\fie(x))={\wt\mu_d}\inv(\fie \inv(\fie(x))),$$
which has cardinality $d^{2q}m_{L|T}$.  Since $\nu \le m_{L|T}$, we conclude that $\nu=m_{L|T}$,  ${B\ud}_{|T\ud}={\wt\mu_d}\inv(\fie \inv(\fie(x)))$ and $B_{|T}=\fie\inv(\fie(x))$. Summing up, we have proven that for general  $x\in T$  the fiber $\fie\inv(\fie(x))$ is mapped to a point by $\fie_{\alpha}$ for every $\alpha\in U_{\rk}$.
This proves that $\fie_{\alpha}$ is composed with $\fie$ for every $\alpha \in U_{\rk}$ (cf. Corollary \ref{cor: rational}). In particular,  for $\alpha\in U_{\deg}$ the map $\fie_{\alpha}$ is birationally equivalent to $\fie$.
\smallskip

(b)   Let $\cP$ be the pull-back to $X\times U_{\rk}$ of the Poincar\'e line bundle on $A\times \Pic^0(A)$ and set $\cV:={\pr_2}_*(\cP\otimes \pr_1^*L)$, where $\pr_i$ is the $i$-th projection, $i=1,2$.
Let $\nu_T\colon  T\unu\to  T$ and $\nu_Z\colon  Z\unu\to Z$ be the normalization maps, and let $\fie\unu\colon T\unu\to Z\unu$ be the  morphism induced by $\fie$. Let $T\unu \to \bar T\overset{\bar \fie} \to Z\unu$ be the Stein factorization of $\phi\unu$: since $|L\otimes \alpha|$ is free for every $\alpha\in \Pic^0(A)$ the map $\fie_{\alpha}$ descends to  $\bar T$ for every $\alpha$ (cf. Lemma \ref{lem: factors}).  Now the  natural  morphism $X\times U_{\rk} \to \pp(\cV) $ induces    a morphism $T\unu\times U_{\rk}\to \pp(\cV) $ that, again by Lemma   \ref{lem: factors}, descends to a morphism
 $\Phi\colon \bar T \times U_{\rk} \to \pp(\cV) $;   we denote by  $\cZ$ the image of $\Phi$. Consider the map  $F\colon \bar T \times U_{\rk}\to Z\unu\times U_{\rk}$ defined by $F(x, \alpha)=(\bar \fie(x),\alpha)$. We have shown in the proof of  (a) that for $x\in\bar  T$ general  the morphism    $\Phi$ contracts $F^{-1}(F(x,\alpha))$  to a point for all  $\alpha \in U_{\rk}$. Since $F$  is a finite morphism and $Z\unu\times U_{\rk}$ is a normal variety,    Lemma \ref{lem: fGalois} and Lemma \ref{lem: factors} imply  that  $\Phi$ descends to a   morphism  $j\colon Z\unu\times U_{\rk}\to \cZ$. Comparing degrees we see that $j$ is birational.

Let $N\in {\rm Pic}(Z)$ be  such that $L=\fie^*N$ and let $\epsilon \colon \widehat Z \to  Z\unu$ be a desingularization.  Pulling back  ${j}^*\cO_{\pp(\cV)}(1)\otimes \pr_1^*({\nu_Z}^* N\inv )$ to $\widehat Z\times U_{\rk}$ we obtain a line bundle that defines  a morphism  $U_{\rk}\to \Pic^0(\widehat Z)$, which in turn extends to a homomorphism $\psi\colon \Pic^0(A)\to \Pic^0(\widehat Z)$. We define the map $g\colon Z\to 	 A$ as the  composition  of  $(\epsilon \circ \nu_Z)\inv $ with the Albanese map $\widehat Z\to \Alb(\widehat Z)$ followed by ${}^t\psi\colon  \Alb(\widehat Z)\to A$.
We claim that $a_{|T}$ and $g\circ \fie$ differ by  a translation in $A$. Without loss of generality, we may replace $T$ by a smooth variety
$T'$ birational to it such that the induced map  $\fie' \colon T'\to \widehat Z$ is a morphism.

 By the universal property of the Albanese map, the  strongly  generating  map $a'\colon T' \to A$  induced by $a_{|T}$  is determined, up to translation in $A$, by  ${a'}^*\colon \Pic^0(A)\to  \Pic^0(T')$,  and it is clear from  the construction of $g$ that $(g\circ \fie')^*={a'}^*$.
\end{proof}
\begin{proof}[Proof of Thm. \ref{thm:factorization}]
As in \S \ref{ssec:  continuous-res}, up to blowing up $X$  we can consider the decomposition $L=P+D$ of $L$ as the sum of  its continuous moving and fixed parts. By  \S \ref{ssec: continuous-res}, taking $d\gg 0$   we have that for $\alpha$ general  $|L\ud\otimes \alpha|=|P\up{d}\otimes \alpha|+D\ud$ and  $|P\ud\otimes \alpha|$ is base point free for all $\alpha$. In addition, by Lemma \ref{lem:ev-degree} we may assume that $m_{L|T}(d)=m_{L|T}$.

 Up to twisting $L$ by a very general element of $\Pic^0(A)$, we may assume that $0\in U_{\deg}\ud$  for every integer $d\ge 1$. By Proposition \ref{prop:factorization},  we may choose  $d_0\gg 0$  such that:
 \begin{itemize}
 \item[--]  $|P\up{d_0}|_{|T}$ induces a  surjective morphism  $\fie\upq{d_0}\colon T\up{d_0}\to Z\upq{d_0}$ of  degree $m_{L|T}$,  birationally equivalent to the morphism  given by $|P\up{d_0}\otimes {\alpha}|_{|T}$ for every $\alpha\in U_{\deg}\up{d_0}$;
 \item[--] ${a_{d_0}}_{|T\up{d_0}}=g\circ \fie\upq{d_0}$ for some map $g\colon Z\upq{d_0}\to A$.
 \end{itemize}

Let $G$ be the Galois group of $\mu_{d_0}$ and $\wt{\mu_{d_0}}$; the natural linearization of  $P\up{d_0}={\wt \mu_{d_0}}^*(P)$  descends to a $G$-action  on  the image $Z\upq{d_0}$ of $\fie\upq{d_0}$ and therefore $\fie\upq{d_0}$ induces a morphism $\fie\colon T\to Z:=Z\upq{d_0}/G$.
So we have a commutative diagram:
  \begin{equation}\label{diag:2}
  \begin{CD}
T\up{d_0}@>\wt{\mu_{d_0}}>>T\\
@V{\fie\upq{d_0}}VV@VV\fie V\\
Z\upq{d_0} @>>>Z\\
@VVV@V VV\\
A@>\mu_d>>A
  \end{CD}
  \end{equation}
By Lemma \ref{lem:galois} (applied to $P\up{d_0}$),  the group  $G$ acts freely on $Z\upq{d_0}$, and so  $\deg \fie=\deg\fie\upq{d_0}=m_{L|T}$.
  Using  the commutativity of  the above diagram, the fact that $H^0(X,L\otimes\alpha)$ is the $G$-invariant subspace of $H^0(X\up{d_0},L\up{d_0}\otimes  \alpha^{d_0})$ and the fact that for general $\beta\in \Pic^0(A)$   the system $|L\up{d_0} \otimes \beta|_{|T\up{d_0}}$ is composed with $\fie\upq{d_0}$, it is easy to check that  $|L\otimes \alpha|_{|T}$ is composed with $\fie$ for general $\alpha \in \pic^0(A)$.  Therefore, by continuity, $|L\otimes \alpha|_{|T}$ is composed with $\fie$  for $\alpha\in U_{\rk}$.

   So far we have proven the existence of $\fie\colon T\to Z$. Next we observe:
      \begin{enumerate}
   \item[(1)] if $\fie\up{d_0}\colon T\up{d_0}\to Z\up{d_0}$ is  the morphism obtained from $\fie \colon T\to Z$ by taking base change with $\mu_{d_0}$, then by the commutativity of the lower square in diagram \eqref{diag:2}, there is an induced  $G$-equivariant map   $Z\upq{d_0} \to Z\up{d_0}$.  Since the maps $Z\upq{d_0} \to Z$ and $Z\up{d_0}\to Z$ are both   \'etale Galois covers  with group $G$, it follows that $Z\upq{d_0} \to Z\up{d_0}$ is an  isomorphism;
   \item[(2)] if $d'=kd_0$ is an integer divisible by $d_0$ and we denote by $\fie\upq{d'}\colon T\up{d'}\to Z\upq{d'}$ the map given by $|L^{(d')}|$ then,  arguing as in (1) and using the isomorphism $Z\upq{d_0} \to Z\up{d_0}$, we see that there is an isomorphism   $Z\upq{d'}\to Z\up{d'}$. 
   \end{enumerate}
By the choice of $L$ and by Proposition \ref{prop:factorization},   the above remarks suffice to prove  that $Z$ satisfies property (b) of the statement for $d\gg 0$. The uniqueness of $\fie\colon T\to Z$ up to birational equivalence follows in a similar way: if $\psi \colon T\to W$ is another map of degree $m_{L|T}$ satisfying properties (a) and (b) of the statement, then for $d$ sufficiently large and divisible there is a birational isomorphism $W\up{d}\to Z\up{d}$ which is compatible with the action of the Galois group of $\wt{\mu_d}$, and therefore descends to a birational  isomorphism $W\to Z$.

Finally, take any integer $k$ and set  $d=kd_0$.
Denote by $G$ the Galois group of $\mu_d$ and by $H$ the Galois group of $\mu_{d_0}$ and denote by $\psi\colon T\up{k}\to Y:=Z\upq{d}/H$  the map induced by $\fie\upq{d}$. Then arguing as above one  shows that the map given by $|L^{(k)}\otimes \alpha|$ is composed with $\psi$ for $\alpha\in\Pic^0(A)$ general.

These remarks complete the proof.
\end{proof}

\section{The continuous rank function}\label{sec: rank}

We use freely the notation introduced in \S \ref{sec: preliminaries}.
Throughout all the section we fix:
\begin{itemize}
\item   a normal  variety $X$ of dimension $n$  and a line bundle  $L$ on $X$;
\item  a morphism $a\colon X\to A$ to an abelian variety  of dimension $q>0$;
\item an ample divisor $H$ on $A$; as usual, we write  $M:=a^*H$, and $M_d=a_d^*H$;
\item a  subvariety $T\subseteq X$ 
 such that $a_{|T}$ is strongly generating (the case $T=X$ is our main interest).
\end{itemize}
Given a line bundle $L$ on $X$    we define $L_x$ as the $\R$-divisor $L+xM$ on $X$.

\subsection{Continuous continuous rank}

We are going to  extend the definition of the restricted continuous rank $h^0_a(X_{|T},L_x)$  for $x\in \Z$ to a function $\phi_T(x)$ for $x\in \R$  which is continuous and    non decreasing.  In a sense, we construct a {\it continuous} continuous rank. When $X=T$ this function is convex and we are able to compute  its left derivative explicitly. \smallskip

\begin{defn}[Extension to $\Q$]\label{defn: Q-ext} Let $T\subseteq  X$  be subvariety such that $a_{|T}$ is strongly generating;  given $x\in\Q$,
choose $d\in \N_{>0}$ such that  $x=\frac{e}{d^2}       $ with $e\in\Z$. We define:
\begin{equation} \label{Q-extension}
h^0_a(X_{|T},L_x):=\frac{1}{d^{2q}}h^0_{a_d}({X\up{d}}_{|T\ud},L^{(d)}+eM_d).
\end{equation}
Note that the definition does not depend on the choice of $d$ by Proposition \ref{prop: mult-rank}.
\end{defn}

The main result of this section is the following:
\begin{thm}\label{thm: main-rank}
Let $X$ be a normal  variety of dimension $n$, $a\colon X \to A$ a   morphism  to an abelian variety of dimension $q>0$ and let   $L$ be a line bundle on $X$.  Then:
\begin{enumerate}
\item   if   $T\subseteq  X$ is  subvariety such that $a_{|T}$ is strongly generating,  then the function $h^0_a(X_{|T},L_x)$, $x\in\Q$,  extends to a continuous  non decreasing  function  $\phi_T\colon\R\to \R$ which has one-sided derivatives at every point $x\in\R$ and is differentiable except at most at countably  many points;
\item the function  $\phi:=\phi_X$  is convex and:
\begin{equation}\label{eq: derivative}
D^-\phi(x)=\lim_{d\rightarrow \infty}\frac{1}{d^{2q-2}}h^0_{a_d}(X\up{d}_{|M_d},(L_x)^{(d)}), \quad \forall x\in\R.
\end{equation}
\end{enumerate}
\end{thm}
\begin{rem}
If   $a_{|T}$ is not strongly generating but it is not constant  and the  maps $a^*\colon \pic^0(A)\to \pic^0(X)$ and  $a_{|T}^*\colon \pic^0(A)\to \pic^0(T)$ have the same kernel, then  (cf. Remark \ref{rem: strongly}) there  is a morphism  $a'\colon X\to A'$, with  $A'$ a positive dimensional abelian variety,  such that $a'_{|T}$ is strongly generating and   $h^0_a(X_{|T},L)=h^0_{a'}(X_{|T}, L)$. So Theorem \ref{thm: main-rank} also holds in this more general situation.  \end{rem}

\begin{rem}
The extended continuous rank function is quite hard to compute in general, even in dimension 1 (see Question \ref{quest: curves}).
In Section \ref{sec: examples} we give some explicit examples, mainly related to abelian varieties (divisors and cyclic coverings) and  curves.

In these examples we can see that the regularity properties of the continuous rank functions can not be improved in general. In Example \ref{ex: jiang}, kindly pointed out to us by Zhi Jiang, the function is not of class ${\mathcal C}^1$, nor convex.

Another property one might expect is that the continuous rank functions are piecewise polynomial, as it happens in all the known examples (see Question \ref{q: polynomial}). Recently Jiang and Pareschi (\cite{JP}), among other interesting properties, proved a partial result in this direction: the continuous rank functions are locally left and right defined by polynomials around rational points.
\end{rem}
We now turn to the proof of Theorem \ref{thm: main-rank}. We start with a  simple calculus result.

\begin{lem}\label{lem: midpoint}
 Let $I \subseteq \R$ be an open interval and let  $f\colon I\cap \Q\to \R$ be a non-decreasing function satisfying the ``midpoint property'', namely such that for every $t,t'\in \Q$ one has
$$\frac{f(t)+f(t')}{2}\ge f\left(\frac{t+t'}{2}\right).$$
Then $f$ extends to a continuous non-decreasing convex function $g\colon I \to \R$.
\end{lem}
\begin{proof}
 For $x\in I$ we define $g(x)=\sup\{f(t)\ |\ t\in\Q\cap I,t\le x\}$. Clearly $g$ extends $f$ and is non decreasing. Observe that for every $x$ we have  $f(x)=\lim_{n\to+\infty}f(t_n)$ for some  sequence of rational numbers $t_n\to x$. It follows that also $g$ has the midpoint property, hence it is enough to prove that $g$ is continuous, since a continuous function with the midpoint property is convex.

Since $g$ is non-decreasing, for every $x\in I$ there exist $g(x)^-:=\lim_{t\to x^-}g(t)$ and $g(x)^+:=\lim_{t\to x^+}g(t)$ and $g$ is continuous at $x$ iff $g(x)^-=g(x)^+$. So assume by contradiction that $g(x)^-<g(x)^+$ and set $\epsilon=g(x)^+-g(x)^-$. For $n\in \N$ large enough we have $g(x+\frac{3}{n})<g(x)^++\epsilon$. Then we have $$g(x-\frac{1}{n})+g(x+\frac{3}{n})< g(x)^-+g(x)^++\epsilon = 2g(x)^+\le 2g(x+\frac{1}{n}),$$
contradicting the midpoint property.
\end{proof}
\begin{lem}
\label{lem: convex}
In the  assumptions of  Theorem \ref{thm: main-rank} the following hold:
\begin{enumerate}
\item The function $h^0_a(X,L_x)$, $x\in\Q$,  extends to a continuous non decreasing convex  function  $\phi\colon \R\to \R$.
 \item The function $h^0_a(X_{|T},L_x)$, $x\in\Q$,  extends to a continuous non decreasing function  $\phi_T\colon\R\to \R$, which is the difference of two continuous convex functions.
 \end{enumerate}
\end{lem}
\begin{proof}
First note that the  functions $h^0_a(X,L_x)$ and  $h^0_a(X_{|T},L_x)$  are  clearly non decreasing for $x\in\Q$.

(i)   By Lemma \ref{lem: midpoint} it is enough to show that $x\mapsto h^0_a(X,L_x)$ has the midpoint property for $x\in \Q$.
So let $x_1<x_2\in \Q$: using a suitable multiplication map we can reduce to the case where $x_1$,  $x_2$  and $t=\frac{x_2-x_1}{2}$ are integers.
Let $L'=L_{x_2}\otimes \alpha$ for $\alpha $ general in $\pic^0(A)$, and take $R_1,R_2\in|tM|$ general  members.
Obviously both $H^0(X,L'-R_1)$ and $H^0(X,L'-R_2)$ are subspaces of $H^0(X, L')$,   and   $H^0(X,L'-R_1)\cap H^0(X,L'-R_2)=H^0(X,L'-R_1-R_2)$ since $|M|$ has no fixed part.
Hence
\begin{gather*}
h^0_a(X,L_{x_2})=h^0(X,L')\geq \dim\left (H^0(X,L'-R_1)+H^0(X,L'-R_2)\right)=\\=h^0(X,L'-R_1)+h^0(X,L'-R_2)-h^0(X,L'-R_1-R_2)=\\
=2h^0_a\left(X,L_{\frac{x_1+x_2}{2}}\right)-h^0_a(X,L_{x_1}),
\end{gather*} which gives the desired inequality.
\smallskip

(ii) Follows from (i)  by  Remark \ref{rem: difference}.
\end{proof}

\begin{lem} \label{lem: Delta}
For all $d\in \N_{>0}$ and all $x\in \R$ one has:
$$
h^0_a(X, L_x)-h^0_a(X, L_{(x-\frac{1}{d^2})})=\frac{1}{d^{2q}}h^0_{a_d}({ X\up{d}}_{|M_d}, (L_x)\up{d}).
$$
\end{lem}
\begin{proof}
All the functions involved are continuous by Lemma \ref{lem: convex}, so it is enough to prove the statement for $x\in \Q$.

So, pick $x\in \Q$ and let $t\in \N$ be such that $(td)^2x=:e$ is an integer.  Since  $(M_d)\up{t}=t^2M_{dt}$, we
 have the following exact sequence on $X\up{dt}$:
\begin{gather}\label{eq: h0a}
0\to H^0(X\up{dt}, L\up{dt}+(e-t^2)M_{dt})\to H^0(X\up{dt}, L\up{dt}+eM_{dt})\to\\
\to  H^0({X\up{dt}}_{|(M_d)\up{t}}, L\up{dt}+eM_{dt})\to 0\nonumber
\end{gather}
Since, possibly after twisting by a very general element $\alpha\in \pic^0(A)$, we may assume that   the $h^0$'s in  \eqref{eq: h0a} are actually $h^0_a$'s,
the claim follows  by  the multiplicative property of the continuous rank functions (Proposition \ref{prop: mult-rank}).
\end{proof}

\begin{proof}[Proof of Theorem \ref{thm: main-rank}]
\noindent(i) Follows from Lemma \ref{lem: convex} and from   the properties of convex functions.

\noindent(ii) Since the left derivative $D^-\phi$ exists at every point by (i), we can compute its value at a point $x$ as the limit of the increment ratio over the   sequence $\{x-\frac{1}{d^2}\}_{d\in \N}$,  so  formula \eqref{eq: derivative} is a consequence of Lemma \ref{lem: Delta}.
\end{proof}

\subsection{Volume and  restricted volume}\label{ssec: reg res vol}

In this section we recall   some known results on the volume, interpreting them in our set-up.
\smallskip

 We keep the assumptions made at the beginning of \S \ref{sec: rank} and we assume in addition:
 \begin{itemize}
 \item $X$ is  smooth;
 \item  $X$ has maximal $a$-dimension; we  denote by $\Sigma\subset X$ the exceptional locus of $a$.
 \end{itemize}
 Under these assumptions we  give  the following:
 \begin{defn}\label{def: a-general}
Let $T\subseteq  X$ be an irreducible subvariety.   We say that $T$ is {\em $a$-general} if it is not contained in $\Sigma$ and the map   $a_{|T}$  is strongly generating. (Note that an $a$-general subvariety $T\subseteq X$ has automatically maximal $a_{|T}$-dimension).\end{defn}
\smallskip

Let $D$ be a big $\Q$-divisor on $X$: we denote
by ${\bf B}_+(D)$ the augmented base locus of $D$, as defined in \cite[10.3]{lazarsfeld-II}. Recall (ibid.)  that  ${\bf B}_+(D)$  depends only on the numerical class of $D$ and that for $0<\lambda\in \Q$ one has ${\bf B}_+(\lambda D)={\bf B}_+(D)$.

For  $L \in \Pic(X)$  integral, denote by ${\rm CB(L)}$ the support of the subscheme $S\subseteq X$,   where  $\mathcal I_SL$ is the image of the continuous evaluation map \eqref{eq: eval-map}; in other words,  ${\rm CB(L)}$ is the set of points of $X$ that belong to the base locus of $|L\otimes \alpha|$ for general $\alpha\in\Pic^0(A)$.  Then we have:
\begin{lem}\label{lem: B+} Let $L\in \pic(X)$ such that $h^0_a(X,L)>0$;  then ${\bf B}_+(L)$ is contained in $\Sigma\cup {\rm CB(L)}$.
\end{lem}
\begin{proof}
Note that  that $L$ is big  by Proposition \ref{prop: big}. By Theorem C of \cite{elmnp}, ${\bf B}_+(L)$ is the union of all the positive dimensional subvarieties $V\subseteq X$ such that $\vol_{X|V}(L)=0$.
Let $V$ be such a subvariety and assume by contradiction that $V$  is not contained in $\Sigma\cup {\rm CB}(L)$. Then $V\to A$ is generically finite onto its image  and $h^0_a(X_{|V},L)>0$. By continuity (cf. Theorem \ref{thm: main-rank}), we have $h^0_a(X_{|V},L-\epsilon M)>0$ for $0<\epsilon \ll 1$.
So, up to replacing  $L$ by a suitable multiple, we may assume  $h^0_a(X_{|V},L-M)>0$.
 Since ${\bf B}_+(L)$ depends only on the numerical class of $L$, up to twisting by    $\alpha\in \pic^0(A)$ we may assume that $h^0_a(X_{|V}, L-M)=h^0_a(X_{|V},L-M)>0$. But then $0=\vol_{X|V}(L)\ge \vol_{X|V}(M)>0$, a contradiction.
 \end{proof}

\begin{prop}\label{lem: derivata volume}Let $T\subseteq X$ be an $a$-general subvariety such that $h^0_a(X_{|T},L_x)\not\equiv 0$.
Set $x_0:=\max\{ x \mid \vol_X(L_x)=0\}$ and  $ \bar x:=\max\{ x \mid h^0_a(X_{|T}, L_x)=0\}$. Then:
\begin{enumerate}
\item   the function $\vol_X(L_x)$, $x\in\Q$,    extends to a continuous function $\psi\colon \R\to \R$, which is differentiable for $x\ne x_0$  and
$$\psi'(x)=\begin{cases}
0 & x<x_0 \\
n\vol_{X|M}(L_x) & x>x_0
\end{cases}
$$
\item the function $\vol_{X|T}(L_x)$, $x\in\Q$, extends to a continuous function  $\psi_T\colon  (\bar x, +\infty)\to\R$, depending only on the numerical class of $L$.
\end{enumerate}
\end{prop}
\begin{proof} (i) The existence of $\psi$ follows from  the continuity of the volume function on $N^1(X)$ \cite[Thm.~2.2.44]{lazarsfeld-I}.
 The function $\psi$ is nondecreasing, hence it is identically $0$  for $x<x_0$; differentiability for $x>x_0$   and the formula  for $\psi'$ are a consequence of Thm.~A and  Cor.~C of   \cite{BFJ} (cf. also \cite[Cor. C]{LM}).
\medskip

(ii) For $x\in (\bar x,+\infty)\cap \Q$ we have that $0<h^0_a(X_{|T}, L_x)=h^0_a(X, L_x)-h^0_a(X, {\mathcal I}_TL_x)$ and so we deduce that $T$ is not contained in $CB(L)$ and that $h^0_a(X, L_x)\neq 0$. Hence, by Lemma \ref{lem: B+} we deduce that $T$ is not contained in  ${\bf B}_+(D)$ and the claim follows from \cite[Thm.~A]{elmnp}.
\end{proof}

We will often use the following remark:
\begin{lem}\label{lem: mult-vol} In the hypotheses  and notation of Proposition \ref{lem: derivata volume}, one has:
$$
\vol_{X\up{d}|T\up{d}}((L_x)\up{d})=d^{2q}\vol_{X|T}(L_x), \quad \forall x >\bar x, \  \forall d\in \N_{>0}
$$
\end{lem}
\begin{proof}  Let $N$ be a line bundle on $X$ such that $h^0_a(X_{|T}, N)>0$.
By Lemma  \ref{lem: B+}  we have $T\not \subseteq {\bf B}_+(N)$ and $\vol_{X|T}(N)$ depends only on the numerical equivalence class of $N$ by \cite[Thm.~A]{elmnp}. Up to twisting by a very general $\alpha\in \pic^0(X)$ we may assume  $h^0(X_{|T},kN)=h^0_a(X_{|T},kN)$ for all $k\in\N$, and therefore $\vol_{X|T}(N)=m!\lim_{k\to\infty}\frac{h^0_a(X_{|T}, kN)}{k^m}$, where $m:=\dim T$.  It follows that
\begin{equation}\label{eq: integral}
\vol_{X\up{d}|T\up{d}}(N)=d^{2q}\vol_{X|T}(N)
\end{equation}  by the multiplicativity of the continuous rank (Proposition  \ref{prop: mult-rank}).

Fix now a rational number $x>\bar x$  and pick $t\in \N_{>0}$ such that $tx$ is an integer. Then \eqref{eq: integral} gives
$$d^{2q}\vol_{X|T}(L_x)=\frac{d^{2q}}{t^m}\vol_{X|T}(t(L_x))=\frac{1}{t^m}\vol_{X\up{d}|T\up{d}}(t(L_x)\up{d})=\vol_{X\up{d}|T\up{d}}((L_x)\up{d}).$$
The claim now extends to all $x\in (\bar x,+\infty)$ by continuity.
\end{proof}

\subsection{Extension of the eventual degree}
\label{ssec: R-eventual degree}
We keep the assumptions made at the beginning of \S \ref{sec: rank} .

 Let $\bar x:=\max\{x\in \R\ |\ h^0_a(X_{|T}, L+xM)=0\}$.
We can extend the definition of eventual degree (cf. \S \ref{ssec: ev-degree}) to line bundles of the form  $L_x$, for $x\in\mathbb{Q}\cap (\bar x,+\infty)$, as we have done for the continuous rank.

Take $d\in \N$ such that $d^2x$ is an integer and define   $m_{L_x|T}:=  m_{L_x\up{d}|T\ud}$.
It is immediate to see that $m_{L_x|T}$ does not depend  on the choice of $d$;
we extend the definition to $x\in \R$ by  setting $m_{L_x|T}=\inf\{m_{L_t|T} \ |\ t\in\Q, \ t\le x\}$.
 The function  $x\mapsto m_{L_x|T}$ is non increasing and  takes only finitely many values (by Theorem \ref{thm:factorization} the
  possible values are the positive  divisors of $\deg a_{|T}$).

Note that even for  $x\in \Q$  there is in general no eventual map associated to $L_x$. Indeed, if  $x=\frac{e}{d^2}$, then the map $a_{d|T^{(d)}}$ has a factorization by a map of degree $m_{L_x|T}$, but this factorization needs  not  descend to $T$.
However, it is possible to say something more in the case when $m_{L_x|T}=2$.
\begin{prop}\label{prop: noncomposto}
Let $\wt T\to T$ be a desingularization and  let $\bar x$ be defined as above.
If   $\wt T$ is of general type and there exist $ x_1<x_2\in (\bar x,+\infty)$ such that $m_{L_{x_1}|T}=m_{L_{x_2}|T}=2$, then $a_{|T}$ is composed with an involution.
\end{prop}
\begin{proof}

Let $C$ be the constant given by Lemma \ref{lem: Gal-a}  for varieties of dimension $m=\dim T$, and let $p>C$ be a prime. For $k$  large enough we may find a rational number $x=\frac{e}{p^k}$ such that $x_1<x<x_2$. In view of the assumptions and of the properties of the extended eventual degree, we have $m_{L_x|T}=2$.
So, setting $d=p^k$,  by Theorem \ref{thm:factorization} the map ${a_d}_{|T\ud}\colon T\ud \to A$ is composed with an involution $\sigma_d$.
In turn, by the choice of $p$ and by Lemma \ref{lem: Gal-a},  the involution $\sigma_d$ induces an involution $\sigma$  of $T$  such that $a_{|T}$ is composed with $\sigma$.
\end{proof}

\section{Castelnuovo inequalities}\label{sec: castelnuovo revisited}

\subsection{Numerical degree of subcanonicity}\label{ssec: subcan}

In  the case of curves, by Clifford's Theorem the ratio of  the degree of a line bundle  to the number of its  global sections is  controlled by the ratio  of   its degree to the degree of the canonical sheaf. In    \cite{barja-severi} the concept of degree of subcanonicity  was introduced in order to formulate and prove  the  Clifford-Severi inequalities, that hold in arbitrary dimension and can be regarded as a vast generalization of Clifford's Theorem.   Here we slightly modify  this definition in order to simplify the proofs, as follows.

\begin{defn}\label{def: nds}
Let $X$ be a smooth variety of dimension $n\geq 2$, with a map $a\colon X \rightarrow A$ to an abelian variety such that $X$ is of maximal $a$-dimension. Let $L\in \pic(X)$ be a big  line bundle.
\begin{itemize}
\item[(i)] Fix $H$ very ample on $A$ and let $M=a^*H$.  We define the {\em numerical degree  of subcanonicity} (with respect to $M$)
$$
r(L,M):=\frac{LM^{n-1}}{K_X M^{n-1}}\in (0,\infty].
$$
\item[(ii)]  We say that $L$ is {\em numerically $r$-subcanonical (with respect to $M$)} if $r(L,M)\leq r$. For simplicity, we say that $L$ is numerically subcanonical if it is numerically $1$-subcanonical.
\end{itemize}
If $n=1$ we define, consistently, $r(L)=\frac{\deg L}{\deg K_X}$.
\end{defn}

\begin{rem}\label{rem: subcan}
The numerical degree of subcanonicity $r(L,M)$ of a big line bundle $L$   has the following properties:
\begin{itemize}
\item If $rK_X-L$ is pseudoeffective (i.e., if $L$ is $r$-subcanonical in the sense of \cite{barja-severi}), then $L$ is numerically $r$-subcanonical.
\item $r(L,M)=\infty$ if and only if $K_XM^{n-1}=0$ and hence if and only if ${\kod}(X)=0$. Indeed, the Kodaira dimension is non negative since $X$ is of maximal Albanese dimension. If $\kod(X)\geq 1$ then  for some $s$ the variety $X$ is  covered by divisors $Q\in |sK_X|$, that are  contracted by $a$ since $H$ is very ample and $QM^{n-1}=0$, contradicting  the assumption that $X$ has maximal $a$-dimension.
\item For  any $d\in \N_{>0}$ we have  $r(L,M)=r(L^{(d)},M^{(d)})=r(L^{(d)},M_d)$. In particular, if $L$ is numerically $r$-subcanonical with respect to $M$, then so is $L^{(d)}$ with respect to $M^{(d)}$ and $M_d$.
\item If $L'\leq L$, then $r(L',M)\leq r(L,M)$ and so if $L$ is numerically $r$-subcanonical then so is $L'$.
\item If $L$ is numerically $r$-subcanonical with respect to $M$, and $M\in |M|$ is a general  divisor, then $L_{|M}$ is numerically $r$-subcanonical with respect to $M_{|M}$ (note that $L_{|M}$ is still big).
\end{itemize}
\end{rem}

\subsection{Continuous  Castelnuovo inequalities}\label{ssec: geografia}

Let us first recall the classical results, which are due mainly to Castelnuovo: see \cite[Chap.III, Sec. 2]{ACGH}.
Let $C$ be a smooth curve and consider a subspace $W\subseteq  H^0(C, L)$ of dimension $r+1\ge 2$ such that the moving part of $|W|$ is a base point free $g^r_d$ (hence $\deg L\geq d$).
Consider the multiplication map
$$
\rho_k\colon \sym^k W \longrightarrow H^0(C, kL).
$$
Let $s:= \left[ \frac{d-1}{r-1}\right]$.
 If  $|W|$ induces a birational morphism, then we have the following estimate on the rank of $\rho_k$ \cite[Lem.~on~p. 115]{ACGH}:
 \begin{equation}\label{eq: castelnuovo}
\rk \rho_k-1\geq \sum_{l=1}^s(l(r-1)+1)+ \sum_{l=s+1}^kd.
\end{equation}
This implies the following inequalities:
\begin{lem}\label{lem: castelnuovo}
Let $C$ be a smooth curve of genus $g$ and $|W|$ a linear system whose  moving part is a base point free $g^r_d$ on $C$. If $k\ge 2$ is an integer,  then
\begin{enumerate}
\item $\rk \rho_k\geq kr+1$;
\item If $|W|$ induces a birational morphism and $d\leq 2g-2$, then
$$\rk \rho_k\geq (2k-1)r.$$

\end{enumerate}
\end{lem}

\begin{proof}
In the birational case both inequalities follow immediately from equation \eqref{eq: castelnuovo} observing that $d\geq r$ and that for  $d \leq 2g-2$ one has  $d\geq 2r$,  and therefore  $s\ge 2$.

For (i), let us suppose that the linear series $|W|$ is not birational.
Then the morphism it induces  factors through a finite map of degree $\geq 2$.
Let $\Gamma$ be the normalization of the image and $j\colon \Gamma \rightarrow \pp^r$ the induced birational morphism.
Then for any $k\geq 2$ the image of $\rho_k$ coincides with the image of the $k$-th multiplication map for the subspace  of $H^0(\Gamma, j^*\cO_{\pp^r}(1))$ given by $W$,
 so the result follows from the birational case.
\end{proof}

We now extend  Castelnuovo's  result to the continuous setting in dimension 1 and 2; these results  are used in Section \ref{sec: C-S revisited}. In a forthcoming paper we will study  extensions to arbitrary dimension.

\begin{prop}\label{prop: continuous castelnuovo per curve}
Let $X$ be a smooth variety of dimension $n\geq 1$ with a  map $a\colon X \rightarrow A$  to  an abelian variety of dimension $q$ and let $C\subseteq  X$ be  a smooth  curve such that $a_{|C}$ is strongly generating.
Let $L\in \pic(X)$, and $k \in \N$. Then
\begin{enumerate}
\item $h^0_a(X_{|C},kL) \geq k\,h^0_a(X_{|C},L).$
\item If $L_{|C}$ is numerically subcanonical  and $\deg a_{|C}=1$, then $h^0_a(X_{|C},kL) \geq (2k-1)\,h^0_a(X_{|C},L).$
\end{enumerate}
\end{prop}
\begin{proof}

\noindent(i) Since the inequality is obviously satisfied if $h^0_a(X_{|C},L)=0$, we may assume that $h^0_a(X_{|C},L)>0$.   For  $d \geq 2$,  set
$
W_d=H_{a_d}^0(X^{(d)}_{|C^{(d)}},L^{(d)})$.
Observe that the image of the  $k$-th multiplication map
$$
\rho_{k,d}\colon{\rm Sym}^kW_d \longrightarrow H^0_{a_d}(C^{(d)},{(kL)^{(d)}}_{|C^{(d)}})
$$

\noindent is  contained in $H^0_{a_d}({X^{(d)}}_{|C\up{d}},kL^{(d)})$.
Then Lemma \ref{lem: castelnuovo} (i) applied to $W_d$, and the multiplicativity of the restricted continuous rank give
\begin{gather*}
d^{2q}h^0_a(X_{|C},kL)=h^0_{a_d}({X^{(d)}}_{|C^{(d)}},(kL)^{(d)})\geq \rk \rho_{k,d} \geq\\
\geq kh^0_{a_d}({X^{(d)}}_{|C^{(d)}},L^{(d)})-(k-1)=d^{2q}kh^0_a(X_{|C},L)-(k-1),
\end{gather*}
and so by letting $d$ go to infinity we obtain
$$
h^0_a(X_{|C},kL)\geq k h^0_a(X_{|C},L),
$$
as wanted.

\medskip

\noindent(ii) The second inequality can be proven in a similar  way.  Again we may assume $h^0_a(X_{|C}, L)>0$, the claim being trivially true otherwise.

Observe that, if $\deg a_{|C}=1$, then by Theorem  \ref{thm:factorization}  we have that
for $d\gg 0$  and $\alpha\in \Pic^0(A)$  general the map induced by the system $|L\otimes \alpha|_{|C}$  on $C$ is  generically injective.
So by Lemma \ref{lem: castelnuovo} (ii) we have the inequality
$$
h^0_{a_d}({X^{(d)}}_{|C^{(d)}},(kL)^{(d)})\geq \rk \rho_{k,d} \geq (2k-1) \,h^0_{a_d}({X^{(d)}}_{|C^{(d)}},L^{(d)})-(2k-1)
$$
 and  we  just take the limit for  $d\to \infty$ as in the proof of (i).
\end{proof}

Now we can deduce the Continuous Castelnuovo inequalities for surfaces:

\begin{thm}\label{teo: higher Castelnuovo}
Let $S$ be a smooth surface with  map $a\colon X \rightarrow A$ to  an abelian variety such that $S$ is of maximal $a$-dimension. Let $L\in \pic(X)$, and $k \in \N$. Then
\begin{enumerate}
\item $h^0_a(S,kL) \geq k^2\,h^0_a(S,L).$
\item If $\deg a=1$ and  $L$ is numerically subcanonical for some $M=a^*H$,  then $h^0_a(S,kL) \geq (2k-1)k\,h^0_a(S,L).$
\end{enumerate}
\end{thm}

\begin{proof}
By Remark \ref{rem: strongly}, it is enough to consider the case when $a$ is strongly generating.

(i) Take a very ample line bundle $H$ on $A$ and let $M=a^*H$.
For  $x\in \R$  set, as usual, $L_x=L+xM$.
Consider the functions
$$
\phi(x):=h^0_a(S,L_x)\,\,\mbox{ and }\phi_k(x):=h^0_a(S,(kL)_x).
$$
Take $x=\frac{e}{d^2}\in\mathbb{Q}$. By Proposition \ref{prop: continuous castelnuovo per curve} (i) applied to $(S^{(d)},M_d,(L_x)^{(d)})$, we have that
\begin{equation}\label{eq: derivate}
h^0_{a_d}({S^{(d)}}_{|M_d},(kL_x)^{(d)})\geq k h^0_{a_d}({S^{(d)}}_{|M_d},(L_x)^{(d)}).
\end{equation}
Since both sides of \eqref{eq: derivate} are continuous functions of $x$, it follows that \eqref{eq: derivate} holds for all $x\in \R$.  Therefore by Theorem \ref{thm: main-rank} we have
$$
D^{-}\phi_k(kx)\geq k D^{-}\phi(x)
$$
for all $x\in \R$.
 Now just  compute
\begin{align*}
h^0_a(S,kL)&=\phi_k(0)=\int_{-\infty}^{0}D^{-}\phi_k(t)dt=k\int_{-\infty}^0D^{-}\phi_k(ky)dy\geq\\
&\geq k^2\int_{-\infty}^0 D^{-}\phi(y)dy=k^2\phi(0)=k^2h^0_a(S,L).
\end{align*}

\medskip

\noindent(ii) Observe that for  a general curve $M_d$ the map ${a_d}_{|M_d}$  is strongly generating  and of degree 1. Hence we can apply inequality (ii)  of Proposition \ref{prop: continuous castelnuovo per curve} in the argument above.
\end{proof}

\begin{cor}\label{cor: severi-3fold}
Let $S$ be a smooth minimal surface, with Albanese map of degree 1. Then
$$
K^2_S \geq 5\, \chi(\omega_S).
$$
\end{cor}
\begin{proof}
Let us apply Theorem \ref{teo: higher Castelnuovo} (ii) to $S$ and $L=K_S$  and to the Albanese map $a\colon S\to A$. We get
$$
\chi(\omega_S^{\otimes 2})=h^0_a(\omega_S^{\otimes 2})\geq 6 h^0_a(\omega_S)=6\chi(\omega_S),
$$
where the last equality follows by the Generic Vanishing Theorem.

\end{proof}

\section{Clifford-Severi inequalities}\label{sec: C-S revisited}

Let $X$ be a smooth projective variety with a map $a\colon X\to A$ to an abelian variety  that is of maximal $a$-dimension. The Main Theorem of \cite{barja-severi}  (Clifford-Severi inequality) is sharp for  a   subcanonical nef line bundle  $L$ on  $X$.
Here  we extend that  result in several ways. First we drop the nefness assumption, working with volumes rather than  with intersection numbers, and   we extend the result to the  relative  setting,  considering  the restricted volume  and the restricted continuous rank (Theorem \ref{teo: Clifford-Severi1}).  In addition, we strengthen the inequality under extra assumptions on the geometry of the map $a$ (Theorem \ref{teo: Clifford-Severi2}). The proofs   are   direct applications of the computations on the derivatives of the restricted continuous rank and volume of Section \ref{sec: rank}.

In this section, we will consider the following

\begin{hyp}\label{hyp: sezzione 6} We assume that $X$ is a smooth variety of dimension $n\geq 1$,     $a\colon X \rightarrow A$  is a map to an abelian variety such that $X$ has maximal $a$-dimension,
 $T\subseteq  X$ is a smooth $a$-general subvariety (Definition \ref{def: a-general})  of dimension $m\ge 1$ and $L\in \Pic(X)$ is a line bundle such that  $T\not\subseteq {\bf B}_+(L)$.
 \end{hyp}

\begin{rem}
Observe that we have the following implications:
$$
h^0_a(X_{|T},L)\neq 0 \Rightarrow T\not\subseteq {\bf B}_+(L) \Rightarrow {\bf B}_+(L)\not=X\, (\mbox{i.e. $L$ is big}),
$$
where the first implication derives from Lemma \ref{lem: B+}.
\end{rem}
\begin{defn}\label{defn: slope}
In the Hypotheses \ref{hyp: sezzione 6}, we define the {\em slope} of $(X,T,L)$ with respect to $a$ as
$$\lambda_T(L):=\frac{\vol_{X|T}(L)}{h^0_a(X_{|T},L)}\in(0,+\infty]$$

In the absolute case, i.e.,  when $T=X$,  we will simply use the notation $\lambda(L)$. Moreover, when $X=T$, $a={\alb}_X$ and $L=K_X$ we  write $\lambda (X)$.
\end{defn}

\begin{rem}\label{rem: lambda-num}
Assume Hypotheses \ref{hyp: sezzione 6}.    The subvariety   $T$ is not contained in ${\bf B}_+(L)$, and so  $\vol_{X|T}(L)$ depends only on the numerical equivalence class of $L$ \cite[Thm.~A]{elmnp}. It follows that   $\lambda_T(L)$ is an invariant of the class of $L$ in $\pic (X)/\pic^0(A)$.
\end{rem}

\begin{rem}\label{rem: vol-mult}
Assume Hypotheses \ref{hyp: sezzione 6}.
If $\eta\colon  X'\to X$ is a birational morphism with $X'$ smooth  such that $T$ is  not contained in the exceptional locus of $\eta\inv$ and $T'\subseteq X'$ is the strict transform of $T$, then we have  $\lambda_T(L)=\lambda_{T'}(\eta^*L)$. Indeed,  the restricted continuous rank does not change, since $X$ is smooth,  and   the restricted volume is  also invariant.

In addition,  we have  $\lambda_T(L)=\lambda_{T\ud}(L\ud)$ for every integer $d>0$ by Proposition \ref{prop: mult-rank} and Lemma  \ref{lem: mult-vol}.
\end{rem}

We will need the following definition:

\begin{defn}\label{defn: delta}
Let $r\in(0,+\infty)$ be a real number.
$$
\begin{array}{rl}
\delta(r)=\begin{cases} 2, & r\leq 1 \\
 \frac{2r}{2r-1},   & r\geq 1
\end{cases} & \quad \quad
\overline{\delta}(r)=\begin{cases}
                                          3 , &  r\leq\frac{1}{2} \\
                                          \frac{6r}{4r-1}, & \frac{1}{2}\leq r\leq 1 \\
                                           \frac{2r}{2r-1}, &  r\geq 1
                                        \end{cases}
\end{array}
$$
Moreover we define, consistently, $\delta(\infty)={\overline \delta}(\infty)=1$.
\end{defn}
The functions  $\delta(r)$ and ${\overline \delta}(r)$ are
 non increasing functions and their  graphs are given in the following figure:

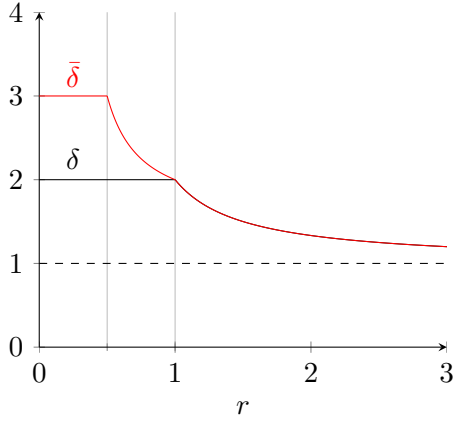
\begin{center}
\begin{tikzpicture}
\begin{axis}[
axis x line=bottom,
axis y line=left,
xmin=0, xmax=3,
ymin=0, ymax=4,
xlabel={$r$},
minor xtick={0.5,1},grid=minor
]

\addplot[
domain=0:1,
samples=200,
]
{2}
[yshift=8pt]
node[pos=0.25] {$\delta$};
\addplot[
domain=1:5,
samples=200,
]
{2*x/(2*x-1)};

\addplot[
domain=0:0.5,
samples=200,
red,
]
{3}
[yshift=8pt]
node[pos=0.5] {$\bar{\delta}$};
\addplot[
domain=0.5:1,
samples=200,
red
]
{6*x/(4*x-1)};
\addplot[
domain=1:5,
samples=200,
red
]
{2*x/(2*x-1)};

\addplot[
dashed,
domain=0:5,
samples=200]
{1};

\end{axis}
\end{tikzpicture}
\end{center}

Now we can state the main theorems of this section:

\begin{thm}[Clifford-Severi inequalities 1]\label{teo: Clifford-Severi1}
Let  $\delta$ and $\overline \delta$ be  the functions of Definition \ref{defn: delta}. Assume that  Hypotheses \ref{hyp: sezzione 6} hold and  that $L_{|T}$ is numerically  $r$-subcanonical with respect to $M_T=a_{|T}^*H$ for some very ample $H\in\pic^0(A)$  (with $r\in(0,\infty]$).  Then

\begin{enumerate}
\item $\lambda_T(L)\ge \delta(r) m!$.
\item  If $T$ is of general type and $a_{|T}$ is not composed with an involution, then
$$\lambda_T(L)\ge {\overline \delta}(r) m!.$$
\end{enumerate}
\end{thm}
\bigskip

\begin{thm}[Clifford-Severi inequalities 2]\label{teo: Clifford-Severinuovo}
In the Hypotheses \ref{hyp: sezzione 6}  one has:
$$\lambda_T(L)\ge  m_{L|T}\,m!, $$
where $m_{L|T}$ is the eventual degree of $L$ with respect to $T$ (Definition \ref{defn: eventual degree}).
\end{thm}

\begin{thm}[Clifford-Severi inequalities 3]\label{teo: Clifford-Severi2}
Assume that   Hypotheses \ref{hyp: sezzione 6} hold,  that $m=\dim T \geq 2$ and that $K_T-L_{|T}$ is pseudoeffective.  Then
\begin{enumerate}
\item If $\deg a_{|T} =1$, then $\lambda_T(L)\ge \frac{5}{2} m!$.
\item If $T$ is of general type and $a_{|T}$ is not composed with an involution, then
$$\lambda_T(L)\ge \frac{9}{4}m!.$$
\end{enumerate}
\end{thm}

\begin{rem} In the absolute case ($X=T$), the condition that  $a$ be strongly generating is not necessary and it is enough to ask that $X$ be of maximal $a$-dimension (cf. Remark \ref{rem: strongly}).

The inequality in Theorem \ref{teo: Clifford-Severi1} (i) is the generalization of the Main Theorem in \cite{barja-severi} for  restricted volume and continuous rank. The inequality in Theorem \ref{teo: Clifford-Severi2} (ii) is a generalization of an inequality given by Lu and Zuo for $n=m=2$  and $L=K_X$ in \cite{LZ2}.
The approach of  Lu and Zuo is based on the  analysis of the (second) relative Noether multiplication map on a fibered surface; their thecnique -similarly to what happens for Xiao's method- is replaced here by the operation of taking  the sum of two integrals.
\end{rem}

\begin{rem} In the proof of Main Theorem in \cite{barja-severi}, it is proven how to reduce the case of non-maximal $a$-dimension to the maximal one under certain hypothesis. With the same argument, we can strengthen Theorem \ref{teo: Clifford-Severi1} (i) to the more general inequality
$$\lambda_T(L)\ge \delta(r) k!$$
\noindent where $k$ is the $a$-dimension of $X$, provided the continuous moving divisor $P$ of $L$ is $a$-big.
\end{rem}
For the case of the canonical line bundle we can even extend this result to the (singular) minimal setting and obtain:

\begin{cor}\label{Cor: Severi Canonico}
Let $X$ be a complex projective minimal $\mathbb{Q}$-Gorenstein variety of dimension $n\ge 2$,  let $\alb_X\colon X\rightarrow \Alb(X)$  be the Albanese map and let $\omega_X={\mathcal O}_X(K_X)$ be the canonical  sheaf. If $X$ is of maximal Albanese dimension, then
\begin{enumerate}
\item $K_X^n\geq 2\,n!\,\chi(\omega_{X})$.
\item If $\deg \alb_X=1$, then $K_X^n\geq \frac{5}{2}\,n!\,\chi(\omega_{X})$.
\item If $\alb_X$ is not composed with an involution, then $K_X^n\geq \frac{9}{4}\,n!\,\chi(\omega_{X})$.\end{enumerate}
\end{cor}

\begin{proof}
 The content of (i) is just  Corollary B of  \cite{barja-severi}. Consider a desingularization $\sigma\colon X'\rightarrow X$. Then we have $\vol_{X'}(K_{X'})=\vol_X(K_X)=K_X^n$. Since the singularities of $X$ are rational and $\omega_X$ is the dualizing sheaf of $X$, we also have that $h^0_a(X',K_{X'})=\chi(\omega_{X'})=\chi(\omega_X)$. Then we can apply
Theorem \ref{teo: Clifford-Severi2}.
\end{proof}

Theorem \ref{teo: Clifford-Severinuovo} is a direct consequence of the properties of the eventual factorization (cf. \S \ref{sec: eventual}) and of Theorem \ref{teo: Clifford-Severi1}, and we prove it first.
\begin{proof}[Proof of Theorem \ref{teo: Clifford-Severinuovo}]
Since the claim is trivial if $h^0_a(X_{|T},L)=0$, from now on  we  assume   $h^0_a(X_{|T},L)>0$.

By  Remark \ref{rem: vol-mult} the slope $\lambda_T(L)$ does not change if:
\begin{itemize}
\item[--]  we  replace   $X$ by a smooth modification such that $T$ is not contained in the exceptional locus of the inverse  map and   we replace  $T$ by its strict transform;
\item[--]  we replace $(X,T)$ with $(X\up{d},T\ud)$ for some   $d\gg 0$.
\end{itemize}
   Hence (cf. \S \ref{ssec: continuous-res}) we may assume   that:
\begin{itemize}
\item $L=P+D$ is the decomposition of $L$ as the sum  of  the continuous moving and fixed part of $L$;  so $|P\otimes\alpha|$ is base point free for all $\alpha\in\Pic^0(A)$  and  $|P\otimes\alpha|$ is the moving part of $|L\otimes\alpha|$ for $\alpha$ general;
\item $T$ is not contained in the support of  $D$;
\item $m_{L|T}(1)=m_{L|T}$ (recall that $m_{L|T}<+\infty$ by Proposition \ref{prop: big});
\item up to twisting by a general  element of $\pic^0(A)$, we may assume that  $h^0(X_{|T},L)=h^0_a(X_{|T},L)=h^0(X_{|T},P)$ and  that $|L|_{|T}$ induces the eventual map (cf. Def. \ref{def:ev-map}) $\fie \colon T\to Z$, which is a morphism of degree $m_{L|T}$.
\end{itemize}
{\bf Caution:}   after these reduction steps  the pair $(X,T)$ still satisfies  Hypotheses \ref{hyp: sezzione 6}, except for the fact  that $T$ may not be smooth. We are going to bypass this problem by working on a desingularization of $T$.
\smallskip

Observe that $T$ is not contained in ${\bf B}_+(P)$ by Lemma \ref{lem: B+},  so by  \cite[Cor.~2.17]{elmnp} we have
 $$ \vol_{X|T}(P)=P^mT=\vol_T(P_{|T}).$$

Now we replace $Z$ by a smooth model $\wt Z$ and $T$ by a smooth model $\wt T$ such that the induced map $\wt T\to \wt Z$ is a morphism.  We abuse notation and we denote by the same letter line bundles on $T$, $Z$ and their pullbacks to  $\wt T$, $\wt Z$.
Since $T$ is not contained in the support of $D$ and since  the volume of a line bundle is invariant under pull back via a birational morphism \cite[Prop.2.2.43]{lazarsfeld-I}, we have
 $$
 \vol_{X|T}(L_{|T})\geq \vol_T(P_{|T})=\vol_{\wt T}(P_{|T})=
   m_{L|T}\vol_{\wt Z}(N), $$
   where   $N$ is  a line bundle on $Z$ such that $P_{|T}=\fie^*(N)$
 and the last inequality holds  by  \cite[Lemma 3.3.6]{holschbach}.

    Finally, $N$ pulls back to a nef line bundle  on  $\wt Z$, so by Theorem \ref{teo: Clifford-Severi1}, we have:
   $$ \vol_{\wt Z}(N)\ge  m! h^0_{\wt a}(\wt Z, N)$$
   where we denote by $\wt a\colon \wt Z\to A$ the induced map.
   Since $h^0_{\wt a}(\wt Z, N)\ge h^0_a(X_{|T},P)=h^0_a(X_{|T},L)$, the claimed inequality follows from the ones above.
 \end{proof}
Theorems \ref{teo: Clifford-Severi1} and \ref{teo: Clifford-Severi2} are  proven by induction, using  results on  linear  systems on  curves and surfaces as the first step.
\begin{lem}\label{lem: 1-previa a teorema C-S}
Let $C$ be a smooth  curve of genus $g(C)\ge 2$ and let  $a\colon C \rightarrow A$ be  a strongly generating map to an abelian variety of dimension $q>0$. Let $L$ be a line bundle of positive degree on $C$. Write $s:=r(L)=\frac{\deg L}{2g(C)-2}$. Then:
\begin{enumerate}
\item $\lambda(L)\geq \delta(s).$
\item Assume that $m_L\neq 2$.
Then
$$
\lambda(L) \geq {\overline \delta}(s).
$$
    \end{enumerate}
\end{lem}

\begin{proof}
The volume of a positive line bundle on a curve is just its degree and so $\lambda(L)=\frac{\deg L}{h^0_a(C,L)}$.
The results are trivially true if $h^0_a(C,L)=0$, so we  assume  $h^0_a(C,L)\neq 0$.

\noindent(i) If $s>1$, then $\vol_C(L)=\frac{2s}{2s-1}h^0(C,L)=\delta(s)h^0_a(C,L)$  by  Riemann-Roch.
If $s\leq 1$, then by Clifford's Theorem  applied  to $L^{(d)}$ on $C\up{d}$ we have
\begin{gather*}
d^{2q}\vol_C( L) = {\vol}_{C\up{d}}(L^{(d)}) \geq 2h^0(C\up{d},L^{(d)})-2\geq \\ \geq 2h^0_{a_d}(C\up{d},L^{(d)})-2=2d^{2q}h^0_a(C,L)-2.
\end{gather*}
Taking  the limit for  $d\to \infty$ we have the desired inequality.
\medskip

\noindent(ii) If $m_L\geq 3$, then  we have  ${\vol}_C(L)\geq 3 h^0_a(C,L)\ge \bar\delta(s)h^0_a(C,L)$  by Theorem \ref{teo: Clifford-Severinuovo}.

Otherwise, by assumption, we have that $m_L=1$.
Hence,  for $\alpha\in \Pic^0(A)$ general and $d\gg 0$  the line bundle $L^{(d)}\otimes \alpha$ induces   a birational map on $C\up{d}$ by Theorem \ref{thm:factorization}.
So, since $\lambda(L)=\lambda(L\ud)=\lambda(L\ud\otimes \alpha)$  (cf. Remarks \ref{rem: vol-mult} and \ref{rem: lambda-num}), we may assume that $|L|$ induces a birational map on $C$.

Assume  $\frac{1}{2}<s\leq 1$. In this case the divisor  $2L$ is non-special and  by Riemann-Roch and the continuous version of Castelnuovo's inequality (Proposition \ref{prop: continuous castelnuovo per curve})  we have:
\begin{align*}
2\vol_C(L)&=2\deg L =h^0_a(C,2L)+(g(C)-1)\geq 3 h^0_a(C,L)+\frac{\deg L}{2s}\end{align*}
 and the inequality $\deg L\geq \frac{6s}{4s-1}\, h^0_a(C,L)=\bar \delta(s) h^0_a(C,L)$ follows.

\medskip
If $s\leq \frac{1}{2}$ we  argue as in the proof of (i),  since if $m_L=1$ the inequality $\deg L\geq 3 h^0_a(C,L)$ holds by  applying the so-called Clifford+  Theorem \cite[III.3.~ex.B.7]{ACGH} to $(C\up{d},L^{(d)})$, and the results follows by taking the limit for $d\to +\infty$.
\end{proof}

\begin{proof}[Proof of  Theorem \ref{teo: Clifford-Severi1}]

Since the claim is trivial for $h^0_a(X_{|T},L)=0$, we assume $h^0_a(X_{|T},L)>0$.
We observe that, as in the proof of Theorem \ref{teo: Clifford-Severinuovo}, we can make a reduction to the absolute case $X=T$. Indeed, by Remark \ref{rem: vol-mult}  we may assume that we have applied a blow-up and a base change by a multiplication map as in \S \ref{ssec: continuous-res}. Hence, the continuous decomposition $L=P+D$ verifies:

\begin{itemize}
\item $|L\otimes\alpha|=|P\otimes\alpha|+D$ for $\alpha\in\Pic^0(A)$ general;
\item $|P\otimes \alpha|$ is base point free for every  $\alpha\in\Pic^0(A)$;
\item $T$ is not contained in the support of $D$.
\end{itemize}
In this case we have that
$$
\vol_{X|T}(L)\geq \vol_{X|T}(P)=P^mT=\vol_T(P_{|T}),
$$
where the first equality follows from  Lemma \ref{lem: B+} and \cite[Cor.~2.17]{elmnp} since $h^0_a(X_{|T},L)>0$,   $T$ is not contained in $\Sigma$ and $|P|$ is base point free.  Observe also that by construction we have $h^0_a(X,L)=h^0_a(X,P)$ and, since $T$ is not contained in $D$,  we also have
$$
h^0_a(X_{|T},P)=h^0_a(X_{|T},L),
$$
so it is enough to prove the inequality $\lambda(P_{|T})\ge\delta(r) m!$ in case (i) and the inequality $\lambda(P_{|T})\ge\bar \delta(r) m!$ in case (ii).

Since we passed to a birational modification of $X$ and replaced $T$ by its strict transform, it is possible that the ``new'' $T$ is not smooth. We get around this issue by considering a resolution $\wt T\to T$ and replacing $P_{|T}$ by its pull back to $\wt T$, which we denote by the same symbol, since:
\begin{itemize}
\item  we have $\lambda_{\wt T}(P_{|T})\le \lambda_T(P_{|T})$ since if we pull back a line bundle  via a  birational morphism the volume does not change \cite[Prop.2.2.43]{lazarsfeld-I}, while  the continuous rank does not decrease (Remark \ref{rem: continuous rank non normale});
\item  the numerical degree of subcanonicity on $\wt T$ of the pull back of $L_{|T}$ is the same as the degree $s$ of subcanonicity of $L_{|T}$ on the ``original'' $T$;
\item the numerical degree of subcanonicity of the pullback of $P_{|T}$ to $\wt T$ is $\le s$.
\end{itemize}
Now the proof works  by induction on $m=\dim T$.

Assume first that $m=1$. First of all observe that  $L_{|T}$ is $r$-subcanonical if and only if $s:=r(L)\leq r$. Since the functions $\delta$ and $\overline \delta$ are non increasing, it is enough to prove the inequality for $s$.
Furthermore,  in  case (ii) if $a_{|T}$ is not composed with an involution then the  eventual degree of any line bundle  is different from  2. As explained above, it is enough to prove the claim in  the absolute case $T=X$, where it follows  directly from  Lemma \ref{lem: 1-previa a teorema C-S} (i) and (ii).

Now let's prove  the theorem for $m\geq 2$, assuming  the result  in dimension $m-1$. Again by the reduction to the absolute case we may assume that $X=T$ (and hence $n=m$). We give  the proof of the inductive step in the cases (i) and (ii) of the theorem separately.
\medskip

\noindent {\underline {Case (i)}}.\par
\noindent Consider the functions
$$
\psi(x):={\vol}_X(L_x) \,\, \mbox{ and}\  \,\, \phi(x):=h^0_a(X,L_x).
$$
Let  $M$ be a  very general element of the linear system $|M|$ and let   $M_d$ be a  very general element of the linear system $|M_d|$.
By \cite[Corollary C]{BFJ} (see also Proposition \ref{lem: derivata volume}) we have that
$$
\psi'(x)=n\, \vol_{X|M}(L_x).
$$
for any $x>x_0$, where $x_0=\max\{t\ |\vol_X(L_t)=0\}$.

By Theorem \ref{thm: main-rank} we have

$$
D^-\phi(x)=\lim_{d \rightarrow \infty} \frac{1}{d^{2q-2}}h^0_{a_d}({X\up{d}}_{|M_d},(L_x)^{(d)}).
$$
We are going to prove that
$$
\psi'(x)\geq n!\delta(r) D^-\phi(x).
$$
Since $\psi'$ is continuous (Proposition \ref{lem: derivata volume} ) and $D^-\phi$ is non decreasing, it is enough to prove the inequality  for rational values of $x$.  Let $\bar x\in \R$ be    the maximum of $\{x\ |\ h^0_a(X, L_x)=0\}$;  then  $x_0\le \bar x$ by Proposition  \ref{prop: big}  and    the above inequality  is trivially true for $x\le \bar x$ by Lemma \ref{lem: derivata volume}.
So fix a rational $x>\bar x$.   Since both $\psi'$ and $D^-\phi$ are multiplicative with respect to base change by multiplication maps (Lemma  \ref{lem: mult-vol} and Proposition   \ref{prop: mult-rank}), we can assume that $L_x$ is integral; in addition, by Remark \ref{rem: vol-mult} we may assume that we have a decomposition $L_x=P_x+D_x$ as in \S \ref{ssec: continuous-res}, where $P_x$ is the  continuous moving part. Then we have the following chain of inequalities
\begin{gather}\label{eq: vol}
\vol_{X|M}(L_x)\geq \vol_{X|M}(P_x)=P_x^{n-1}M= \frac{1}{d^{2q}}((P_x)^{(d)})^{n-1}M_d,\\
\frac{1}{d^{2q-2}}((P_x)^{(d)})^{n-1}M_d=\frac{1}{d^{2q-2}}\vol_{X\up{d}|{M_d}}({P_x}^{(d)}),\nonumber
\end{gather}

\noindent where the first and the  last equality follow by Lemma \ref{lem: B+} and \cite[Cor.~2.17]{elmnp},  since $|P_x|$ is base point free and $M$ and $M_d$  can be taken to be very general.

Observe also that
\begin{equation}\label{eq: h0aa}
h^0_{a_d}({X\up{d}}_{|{M_d}},(P_x)^{(d)})=h^0_{a_d}({X\up{d}}_{|{M_d}},(L_x)^{(d)}).
\end{equation}
By Remark \ref{rem: subcan}, if $L$ is $r$-subcanonical (with respect to $M$)  so are  $(L_x)^{(d)}$, $(L_x)^{(d)}_{|M^{(d)}}$ and, as a consequence,
 $(P_x)^{(d)}_{|{M_d}}$, for any $d$.
 Hence  if $(X,L,a)$ verifies the hypotheses of the theorem,  then we can conclude that $(M_d,(P_x)^{(d)}_{|M_d},{a_d}_{|M_d})$ verifies the same hypotheses. So    the inductive assumption gives
\begin{equation}\label{eq:  ind}
\vol_{X\up{d}|M_d}(P_x\ud)\ge  (n-1)!\delta(r) h^0_{a_d}({X\up{d}}_{|{M_d}},(P_x)^{(d)})=(n-1)!\delta(r)h^0_{a_d}({X\up{d}}_{|{M_d}},(L_x)^{(d)}).
\end{equation}

Combining \eqref{eq: vol}, \eqref{eq: h0aa}, \eqref{eq: ind}  and taking the  limit over $d$, we have
$${\vol}_{X|M}(L_x)\geq (n-1)!\delta(r)D^-\phi(x)
$$
and so
$$
\psi'(x)\geq n! \, \delta(r) \, D^-\phi(x).
$$
By Theorem \ref{thm: main-rank} the function $\phi$ is convex, and therefore absolutely continuous,  and by Proposition \ref{lem: derivata volume} the function $\psi$ is piecewise of class $C^1$. So we may apply the Fundamental Theorem of Calculus for the Lebesgue integral  and  compute
\begin{align*}
\vol_X(L)=\psi(0)=\int_{- \infty}^0 \psi'(x)dx&\geq \delta(r)\,n!\,\int_{- \infty}^0D^-\phi(x)dx=\delta(r)\,n!\,\phi(0)=\\
&=\delta(r)\,n!\,h^0_a(X,L).
\end{align*}

\noindent {\underline {Case (ii).}}

Let us explain what has to be modified  in the  proof of Case (i) in order to prove the formula in Case (ii). Observe that the general argument works just changing $\delta$ to $\overline \delta$.
However we need to check in addition that whenever we take base change by a multiplication map or restrict to a subvariety the assumption that the variety is of general type and that the  corresponding   map to $A$ is not composed with an involution still holds.
 In order to verify   the  latter condition, let $C$ be the constant associated to the variety of general type  $T$ given by Lemma \ref{lem: Gal-a}. Then consider the set of natural numbers ${\mathcal D}=\{d=p^k\,|\, p>C\,{\rm is}\,{\rm prime}\,\}$. Observe that the set of points $\{ x=\frac{c}{d^2}\,|\,c\in \mathbb{Z},\,d\in {\mathcal D}\,\}$ is dense in $\mathbb{R}$. So, it is enough to apply the density argument of   Case (i) only to  rational numbers of this form. Also, all the limits for $d\in \mathbb{N}$ can be substituted by limits over $d\in {\mathcal D}$ since the limits exist and we are  taking  a subsequence.

Consequently, if $a$ is not composed with an involution, then  neither is $a_d$ for $d\in {\mathcal D}$  by Lemma \ref{lem: Gal-a}. So we can apply Proposition \ref{prop: composto involuzione} and conclude that ${a_d}_{|M_d}$  is not  composed with an involution, either.

The property that  $T$ be of general type is also maintained in all the inductive process since then $T^{(d)}$ is also of general type and any section  is, by adjunction.
\end{proof}
In order to prove Theorem \ref{teo: Clifford-Severi2} we need first to prove the result for surfaces.
Note that point (ii) of the theorem is a generalization of \cite[Theorem 3.1]{LZ2}.

\begin{prop}\label{lem: 2-previa a teorema C-S}
Let $S$ be a smooth surface with $a\colon S \rightarrow A$ a strongly generating map to an abelian variety such that $S$ is of maximal $a$-dimension. Let $L$ be a line bundle on $S$ and assume that  $K_S-L$ is pseudoeffective.
Then:

\begin{enumerate}
    \item If $\deg a=1$, then ${\vol}_S(L)\geq 5 h^0_a(S,L)$.
    \item If $S$ if of general type and $a$ is not composed with an involution, then ${\vol}_S(L)\geq \frac{9}{2}h^0_a(S,L)$.
\end{enumerate}
\end{prop}

\begin{proof}

If $h^0_a(S,L)=0$, the result is trivially true. Otherwise,  arguing as in the reduction process at the beginning of the proof of Theorem \ref{teo: Clifford-Severi1},  we may assume that the linear system $|L|$  is base point free. Since the inequalities we want to prove are invariant under base change by  multiplication maps, by Theorem \ref{thm:factorization} we can assume that the map  induced by $|L|$ is generically finite of degree $m_{L}$. Note that the condition that $K_S-L$ is pseudoeffective is preserved during the reduction process.

Take a general $C\in|L|$  and a general $\alpha\in \pic^0(A)$ and consider the exact sequence
$$
0\longrightarrow H^0(S,L\otimes \alpha)\longrightarrow H^0(S,L+(L\otimes \alpha))\longrightarrow H^0(C,L+(L\otimes \alpha)_{|C}).
$$
Since $L$ is base point free and $K_S-L$ is pseudoeffective, we have that $L^2\leq LK_S$ and so $\deg (2L_{|C})\leq 2g(C)-2$. So we can conclude that
\begin{align}\label{eq: L2}
L^2&=\frac{1}{2} \deg (2L_{|C})\geq h^0_a(C, 2L_{|C})= h^0(C,L+(L\otimes \alpha)_{|C})\geq h^0_a(S,2L)-h^0_a(S,L),
\end{align}
where the first inequality follows from Lemma \ref{lem: 1-previa a teorema C-S} (i).
\smallskip

\noindent (i) Assume that $\deg a=1$. Then by  Theorem \ref{teo: higher Castelnuovo} (ii),  with $k=2$, we have that $h^0_a(S,2L)\geq 6 h^0_a(S,L)$ and the result follows by \eqref{eq: L2}.
\smallskip

\noindent(ii) Assume now that $a$ is not composed with an involution.
 Fix a  very ample $H$ on $A$, let $M=a^*H$ and take a smooth curve $M$ in the linear system $|M|$.

Let $\bar x\leq 0$ be the infimum of the $x\leq0$ such that $m_{L_x}=1$ (take $\bar x=0$ if this condition is empty).

The map $a$ is not composed with an involution. This property is not preserved  by a general base change, but the following claim only needs that the original map $a$ be non composed: by Proposition \ref{prop:  noncomposto} we have that $m_{L_x}\neq 2$  except possibly for the single value $\bar x$. Hence $m_{L_x}\geq 3$ for all $x < \bar x$, and $m_{L_x}=1$ for $x>\bar x$. Observe that, since  we may take $M_d$ general in  $a_d^*|H|$, we also have $m_{L_x\up{d}|M_d}=1$ or $\ge 3$ in the cases $x>\bar x$ or $x<\bar x$, respectively.

Let $\phi$ and $\phi_2$ be the functions defined in the proof   of Theorem \ref{teo: higher Castelnuovo}. If $d\gg 0$ is such that $(L_x)\ud$ is integral, then the  linear systems $|(L_x)\ud |_{|M_d}$ induce birational maps on $M_d$ for $x>\bar x$, and so by a direct combination of the proofs of (i) and (ii)  of Theorem \ref{teo: higher Castelnuovo}, we  obtain that
\begin{gather*}
h^0_a(S, 2L)= \phi_2(0)= \int_{-\infty}^0D^-\phi_2(t) dt=2\int_{-\infty}^{0}D^-\phi_2(2x) dx= \nonumber \\
 = 2\int_{-\infty}^{\bar x}D^-\phi_2(2x) dx+ 2\int_{\bar x}^0D^-\phi_2(2x) dx \geq 2\int_{-\infty}^{\bar x}2D^-\phi(x) dx+ 2\int_{\bar x}^03D^-\phi(x) dx= \nonumber \\
 =4h^0_a(S,L_{\bar x})+6h^0_a(S,L)-6h^0_a(S,L_{\bar x}).\nonumber \\
\end{gather*}
Plugging the above inequality in \eqref{eq: L2}, we get
\begin{equation}\label{formula: 1}
L^2\geq 5h^0_a(S,L)-2h^0_a(S,L_{\bar x}).
\end{equation}
Now we will obtain a second inequality among these  invariants, as follows. Consider now the volume function
$$
\psi(x)=\vol_S(L_x).
$$
Recalling the expression of $\psi'$ given in Proposition \ref{lem: derivata volume},  we apply Theorems \ref{teo: Clifford-Severinuovo}  and \ref{teo: Clifford-Severi1} to $(M_d,(L_x)\up{d}_{|M_d},{a_d}_{|M_d})$ to find inequalities involving $\psi'(x)$,  $D^-\phi(x)$ and  $m_{(L_x)\ud|M_d}$  and  we obtain that
\begin{align}\label{formula: 2}
L^2=&\vol_S(L)=\psi(0)= \int_{-\infty}^0\psi'(x) dx=\int_{-\infty}^{\bar x}\psi'(x) dx+ \int_{\bar x}^0\psi'(x) dx\geq \nonumber \\
&\geq 6\int_{-\infty}^{\bar x}D^-\phi(x) dx+ 4 \int_{\bar x}^0D^-\phi(x) dx =6h^0_a(X,L_{\bar x})+4 h^0_a(S,L)-4h^0_a(X,L_{\bar x})= \nonumber \\
&=4h^0_a(S,L)+2h^0_a(S,L_{\bar x}).
\end{align}
\noindent and so the result follows by adding (\ref{formula: 1}) and (\ref{formula: 2}).
\end{proof}

\begin{proof}[Proof of Theorem \ref{teo: Clifford-Severi2}]
The same argument as in  Theorem \ref{teo: Clifford-Severi1} works, using  induction on $m$ and  taking  the results of Proposition \ref{lem: 2-previa a teorema C-S}  for $m=2$, with $\delta = 5 $ (resp. $\delta=\frac{9}{2}$) in case  (i) (resp. (ii))  as  the starting step.
\end{proof}

\section{Examples}\label{sec: examples}
In this section we make  explicit computations of two kinds.
Firstly, we compute the continuous rank function for some pairs $(X,L)$. As we will see, the computations are non trivial already in the case of curves.
In the second set of examples, we compute the slope of many pairs, mainly using some covering construction. In both cases, the results we obtain naturally lead us to some speculations, which we formulate as open questions.
\subsection{Explicit computations of  the  continuous rank}
\begin{ex}[Divisors in abelian varieties]\label{ex: theta}
Let $A$ be an abelian variety of dimension $q$ and let $Y\subseteq A$ be an ample normal  divisor; we compute $\phi$ for $L=0$ and $H=M=\OO_A(Y)$.
 Fix $0<x\in \Q$ and write $x=\frac{e}{d^2}$ for some integers $e,d$. The divisor  $Y\ud$ is linearly equivalent to $d^2Y+\beta$ for some $\beta\in \Pic^0(A)$, so
 twisting the restriction sequence for $Y\ud$ in $A$  by $eY+\alpha$, where  $\alpha\in\pic^0(A)$ is general, we get:
$$0\to \OO_A((e-d^2)Y+\alpha-\beta)\to \OO_A(eY+\alpha )\to \OO_{Y\ud}(eY+\alpha )\to 0.$$
Using  Kodaira vanishing and the fact that $\alpha$ is general, we get:
$$h^0_{a_d}(Y\up{d}, eX)=h^0(A, eY+\alpha )-h^0(A, (e-d^2)Y+\alpha-\beta).$$
 Setting  $s:=h^0(A,Y)$,  we have:
$$
\phi(x)=\begin{cases}
0, & x\le 0\\
s x^q, & 0<x\le 1 \\
s (x^q-(x-1)^q), & x>1
\end{cases}
$$
So $\phi$ is of class $\cC^{q-1}$ in this case and  coincides with the restricted continuous rank function for $X=A$ and  $T=Y$.
Note  also that for $q=2$ and $s=1$ we obtain the continuous rank function for  a curve of genus 2 with respect to its Abel-Jacobi map.
\end{ex}
\begin{ex}[Double covers of abelian varieties]
If $X\to A$ is a double
cover given by the  relation $2M\equiv B$, where $B$ is a smooth ample divisor, the rank function for $X$ with $L=0$ can be computed arguing as in Example \ref{ex: theta}.
Setting $s:=h^0(A, M)$, and $\phi(x)=h^0_a(X, xM)$, one gets:
$$
\phi(x)=\begin{cases}
0, & x\le 0\\
s x^q, & 0<x\le 1 \\
s (x^q+(x-1)^q), & x>1
\end{cases}
$$
\end{ex}

 \begin{ex}[Non-simple abelian varieties]\label{ex: jiang}This  example, kindly pointed out to us by Zhi Jiang, shows that  the regularity properties of the continuous rank function given in  Theorem  \ref{thm: main-rank} cannot be improved without further assumptions.

 Let $A_1$ be an abelian variety of dimension $q-1>0$, let   $C$ be  an elliptic curve and set $X=A=A_1\times C$. Let $L$ be the pull-back of an ample divisor on $A_1$ and let $M$ be a very ample divisor on $A$.

We denote by $\phi(x)$ the continuous rank function $h^0_a(X, L+xM)$.  We have:
$$\phi(x)=
\begin{cases}
0 & x\le 0\\
\frac{1}{q!}(L+xM)^q= \frac{1}{q!}\sum_{i=1}^q {q\choose i} x^iM^iL^{q-i} & x>0
\end{cases}
$$
Note that $\phi$ is continuous but  not differentiable for $x=0$.

Assume now that $q=2$, denote by $L_1$, resp. $L_2$ the pullback of a point of $A_1$ resp. $C$,  and set $L=L_1$, $M=L_1+L_2$ and take $T$ to be a general element of $|M|$. Then one easily computes

$$\phi_T(x)=
\begin{cases}
0 & x\le 0\\
\frac 12 (L+xM)^2=x+x^2 & 0<x \le 1\\
\frac 12(L+xM)^2-\frac 12(L+(x-1)M)^2 =2 x& x>1
\end{cases}
$$
So the function $\phi_T(x)$ is not convex in this case.
\end{ex}

Computing  continuous rank functions is definitely non trivial even in the case of a curve $C$: actually, for  $g(C)>2$  we are unable to give a complete answer even when $L=\OO_C$ and $a$ is the Abel-Jacobi map. Our observations are summarized in the following:

\begin{prop}\label{prop: rank-curve}
Let $C$ be a smooth curve of genus $g$ and $a\colon C\longrightarrow A=J(C)$ the Abel-Jacobi map.
Consider $L={\mathcal O}_C$, let $H=\Theta$ be the theta line bundle on $A$  and $M=a^*\Theta$  the induced degree $g$ divisor on $C$.
Let $\phi(x)=h^0_a(C,xM)$ be the absolute continuous rank function induced and let ${\widetilde \phi}(x)=h^0_{\rm Id}(A_{|C},x\Theta)$ be the corresponding restricted continuous rank function. Then
\begin{itemize}
\item[(i)] $${\widetilde \phi}(x)=\left\{\begin{array}{ll}
                                   0 & {\rm if}\, x\leq 0 \\
                                   x^g & {\rm if}\, 0\leq x\leq 1 \\
                                   gx+(1-g) & {\rm if}\, x\geq 1
                                 \end{array}\right.$$
\item[(ii)] $\phi(x)={\widetilde \phi}(x)$ for all $x\notin (0,1)$.
\end{itemize}
\end{prop}
\begin{proof}
Clearly we have for all $x\in \mathbb{R}$ that $\phi(x)\geq {\widetilde \phi}(x)\geq 0$, and that both are 0 if $x\leq 0$.

By Riemann-Roch theorem we have that $\phi (x)=gx+(1-g)$ if $x\in \mathbb{Z}$, $x\geq 2$.
Moreover, $\phi(1)=h^0_a(C, \Theta)=1$.
Using that $\phi(x)$ is a convex function, we deduce that
$${\rm for}\,\, x\geq 1\,\,, \phi(x)=gx+(1-g).$$

On the other hand, the behavior of ${\widetilde \phi}(x)$ is easy to compute in the interval $[0,1]$. For this, observe that $h^0_{id}(A, {\mathcal I}_{C,A}(\Theta))=0$. Since the continuous rank function is non-decreasing, we obtain that
$$
{\rm for}\,\, 0\leq x \leq 1\,\,, {\widetilde \phi}(x)=h^0_{\rm Id}(A,x\Theta)=x^g
$$
by the computation in Example \ref{ex: theta}.

In order to finish the proof, we need to obtain that $h^0(C, xM+\alpha)=h^0(A_{|C}, x\Theta+\alpha)$  for a general $\alpha$ and for all $x\geq 1$. Since both  functions are continuous, it is enough to prove it for rational values of $x> 1$. Assume that $x=1+\frac{e}{d^2}>1$ ($e>0$). The result follows from the surjectivity of the map
$$
H^0(A, \Theta ^{(d)}+e\Theta_d+\alpha)\longrightarrow H^0(C^{(d)},M^{(d)}+eM_d+\alpha)
$$
for a general $\alpha$. In order to prove this surjectivity we will prove that $h^1(A, {\mathcal I}_{C^{(d)},A}(\Theta ^{(d)}+e\Theta_d+ \alpha)=0$ for any $e>0$ and any $\alpha$.

By \cite[Lemma 3.3]{PPGVminimal} in  and \cite[Example 3.10]{PPM-regularity} we have that the sheaf ${\mathcal F}={\mathcal I}_{C,A}(\Theta)$ is a GV-sheaf and so ${\rm codim}_AV^i({\mathcal F})\geq i$ for all $i \geq 1$ by \cite[Lemma 3.6]{pp}. Let ${\widetilde {\mathcal F}}={\mathcal F}^{(d)}={\mathcal I}_{C^{(d)},A}(\Theta ^{(d)})$. By the projection formula we have that $V^i({\widetilde {\mathcal F}})$ is a finite union  of translates of $V^i({\mathcal F})$, and hence of the same dimension. Hence, again by \cite[Lemma 3.6]{pp}, we conclude that ${\widetilde {\mathcal F}}$ is a GV-sheaf.
The line bundle $e\Theta_d$ is ample and hence $IT_0$. By \cite[Proposition 3.1]{PP-regularity} we deduce that  ${\widetilde {\mathcal F}}\otimes e\Theta_d$ is also an $IT_0$ sheaf and hence we deduce the vanishing of the higher twisted cohomology.
\end{proof}

\begin{qst}\label{quest: curves}
In Proposition \ref{prop:  rank-curve} we have computed the restricted continuous rank function for a curve $C$ inside its Jacobian, but we have been able to determine the corresponding {\em absolute} continuous rank function $\phi$ only for $x\ge 1$. By Example \ref{ex: theta} these functions coincide when $g(C)=2$, and it would be interesting to know whether this is always true.
At the moment, it is not clear to us even whether $\phi$ depends only on $g$ or also on the curve $C$.
\end{qst}

\begin{qst}\label{q: polynomial}
In all the examples that we  are able to compute the continuous rank function is piecewise polynomial. One wonders whether this is always the case. Recently Jiang and Pareschi (\cite{JP}) obtained a partial result in this direction: given a $x_0\in\mathbb{Q}$, there exist an $\epsilon >0$ and polynomials $P_-(x)$ and $P_{+}(x)$ (depending on $x_0$) such that $\phi (x)$ is given by $P_-(x)$ and $P_+(x)$ in $(x_0-\epsilon,x_0]$ and $[x_0,x_0+\epsilon)$, respectively.
\end{qst}

\subsection{Computation of slopes}
We consider the slope $\lambda(X)$ for smooth varieties $X$ of general type and maximal Albanese dimension (``m.A.d.'') (cf. Definition \ref{defn: slope}). Recall that by the Generic Vanishing theorem \cite{GL}, \cite{pp} in this case $h^0_a(X,K_X)=\chi(X,K_X)$.

Our first examples involve covering constructions and have non-birational Albanese map.

\begin{ex}[Simple cyclic covers]\label{ex:simple-cyclic}

Let $Y$ be a smooth variety of general type and m.A.d., of dimension $n>1$. A simple cyclic cover of $\pi\colon X\to Y$ of degree $t$ is given by   a linear equivalence of the form $tL\sim B$, with $B$ a   smooth effective divisor (cf. \cite[\S~2]{ritaabel}), so that $\pi_*\OO_X=\oplus_0^{t-1}L^{-i}$.  We assume in addition that $B$ is ample. The variety $X$ is smooth, since $B$ is smooth, and the expression  for $\pi_*\OO_X$
and  Kodaira vanishing give  $q(X)=h^1(X,\OO_X)=h^1(Y,\OO_Y)=q(Y)$.  It is easy to check that the Albanese map of $X$ is $a_Y\circ \pi$, where $a_Y$ is the Albanese map of $Y$.
We have  $K_X=\pi^*(K_Y+(t-1)L)$ by the Hurwitz formula and  $\chi(K_X)=\sum_{i=0}^{t-1}\chi(K_Y+iL)$.  The line bundle  $K_Y+(t-1)L$ is big, since $K_Y$ is effective, hence $X$ is  of general type with $\chi(K_X)=\sum_0^{t-1}\chi(K_Y+iL)$. It is immediate to check that  $X$ is minimal if $Y$ is.

If $Y$ is an abelian variety, then one can compute explicitly  $$\lambda(X)=n!\frac{t(t-1)^n}{\sum_0^{t-1}i^n}.$$
So in the case $t=2$ of double covers we obtain $\lambda(X)=2n!$.
For any $t$, we have  $\lambda(X)=6(1-\frac{1}{2t-1})$ for $n=2$  and  $\lambda(X)=24(1-\frac{1}{t})$ for $n=3$.
For any  $n$,  the polynomial    $s(t):= {\sum_0^{t-1}i^n}$ has  degree $n+1$  and can be written as $s(t)=\frac{1}{n+1}t^{n+1}-\frac{1}{2}t^n+o(t^{n})$, hence for fixed $n$   the slope $\lambda(X)$ tends to $(n+1)!$ from below for $t\to +\infty$.

When $Y$ is any variety of m.A.d.,   to compute the slope of $X$  one has to apply the Riemann-Roch theorem on $Y$ and obtains in general  a quite complicated formula. It is easier to understand the asymptotic behavior of $\lambda(X_m)$ when  $L=mH$ with $H$ a fixed ample divisor and $m\gg0$. Consider for simplicity the case $t=2$. If  $X_m\to Y$ is a double cover with   smooth branch divisor $B_m\in |2mH|$, then
$$\vol_{X_m}(K_{X_m})=2\vol(K_Y+mH)=2m^nH^n+o(m^n)$$
where the second equality is implied by the fact that $\vol(K_Y+mH)=m^n\vol(H+\frac{1}{m}K_Y)$ and by the continuity of the volume function.
Since $\chi(X_m)=\frac{m^n}{n!}H^n+o(m^n)$, we get  $$\lim_{m\to +\infty} \lambda(X_m)=2n!.$$

Analogously, for  any  $t>2$ the slope  of a  simple cyclic cover   $X_m\to Y$ of degree $t$ branched on a general  element of $|tL|$  approaches for $m\to \infty$  the value obtained when $Y$ is an abelian variety.
\end{ex}
\begin{ex}[``Small perturbations'' of the slope]\label{ex:small}
The  covering  construction of  Example \ref{ex:simple-cyclic}
involves the choice of an ample divisor $L$.

If one starts with any variety $Y$ of general type and m.A.d. and  performs the construction with the variety  $Y\up{d}$ and the line bundle $L=M_d$ given by the covering trick,  by the multiplicativity properties of the volume and of the continuous rank,  the slope  of the variety $X\up{d}$ that one obtains can be computed as:
\begin{gather*}\lambda(X\ud)=\frac{2\vol(K_{Y\ud}+M_d)}{h^0_{a_d}(Y\ud, K_{Y\ud}+M_d)+h^0_{a_d}(Y\ud, K_{Y\ud})}=\\
\frac{2\vol(K_{Y}+\frac{1}{d^2}M)}{h^0_a(Y, K_Y+\frac{1}{d^2}M)+h^0_a(Y, K_Y)}.
\end{gather*}
So, by the continuity of the volume and continuous rank functions, we obtain $\lim_{d\to\infty}\lambda(X\ud)=\lambda(Y)$.
A similar computations gives the same result  when $X\ud$ is a simple cyclic cover of degree $t$ of $Y\ud$ branched  on a divisor of $|tM_d|$ (cf. Example \ref{ex:simple-cyclic}).

Note that the degree of the Albanese map gets multiplied by  the degree $t\ge 2$ of the cover used in the construction. Hence, given any $Y$ of m.A.d. we can construct $X$ with $\lambda(X)$ arbitrarily close to $\lambda(Y)$ and  Albanese map of arbitrarily large degree.
\end{ex}

\begin{ex}[The slope is unbounded for $n\ge 3$]\label{ex:no-BMY}\label{ex: slope-unbounded}
\par
One can write the  Bogomolov-Miyaoka-Yau inequality for surfaces of general type  in the form  $\lambda(X)\le 9$.
The   BMY type inequality $\lambda(X)\le 72$ for Gorenstein minimal threefolds of general type is proven in \cite{CCZ}.
In dimension $n\ge 3$ there exist $\Q$-Gorenstein varieties of general type and m.A.d. with $\chi(K_X)=0$ \cite[Ex.~1.13]{ein-lazarsfeld}; hence an analogue of the BMY inequality  cannot hold in general.
However one might  hope that for $n$-dimensional varieties of m.A.d.  with $\chi(K_X)>0$  a bound of the form  $\lambda(X)\le C(n)$ hold.
We use the construction of Example \ref{ex:small} to show that this is not the case.
Let $Y$ be a smooth  variety of dimension $n\ge 3$ of general type  and m.A.d. with $\chi(K_Y)=0$  (note that the minimal model of $Y$ is necessarily not Gorenstein).
If  $X\up{d}$ is  constructed  as in Example \ref{ex:small},  one has $\chi(K_{X\ud})>0$ and  $\lim_{d\to+\infty}\lambda(X\up{d})=\lambda(Y)=+\infty$.
\end{ex}

The examples that follow have  Albanese map of degree 1.

\begin{ex}[Complete intersections]\label{ex:complete-int}
If $A$ is an $(n+1)$-dimensional abelian variety and $X\subseteq A$ is a smooth ample divisor, then $X$ is a minimal $n$-dimensional variety satisfying
$\lambda(X)=(n+1)!$.

More generally, let $t>0$ be an integer,  let  $A$ be  an abelian variety of dimension $n+t$, $n\ge 2$,   and let $X\subseteq A$ be a smooth complete intersection of $t$ divisors in $|L|$, where $L\in \Pic(A)$ is ample. By adjunction we  have $K_X=tL_{|X}$, hence   $X$ is minimal   of general type.
By Lefschetz theorem the inclusion  $X\into A$ is the Albanese map of $X$.
 Standard computations give
 $$\lambda(X)=\frac{(n+t)!t^n}{ \sum_{i=0}^{t-1} (-1)^{i}{t\choose i}(t-i)^{n+t}}.$$
 For $t=2$, this gives $\lambda(X)=\frac{(n+2)!2^{n-1}}{2^{n+1}-1}> \frac{(n+2)!}{4}\ge (n+1)!$ and $\lambda(X)=\frac{48}{7}$  for $t=n=2$.

 For $t=3$ it gives $\lambda(X)=\frac{(n+3)!3^{n-1}}{3^{n+2}-2^{n+3}+1}>(n+1)!$ and $\lambda(X)=\frac{36}{5}$  for $t=3$, $n=2$.

 As in Example \ref{ex:simple-cyclic}, replacing $A$ by a smooth $(n+t)$-dimensional  variety $Y$ of m.A.d., taking  $X$ a smooth complete intersection of $t$ divisors in    $|mH|$ where $H$ is an ample divisor and letting $m$ go to infinity we obtain values of $\lambda(X)$ approaching the corresponding values for the case when $Y$ is an abelian variety.
\end{ex}
\begin{ex}[Products and symmetric products of curves]\label{ex:curves}
Let $n\ge 2$ be an integer. If $X$ is the product of $n$ curves of genus $\ge 2$, then $\lambda(X)=2^nn!$.

Let now $C$ be a curve of genus $g>0$ and let $X:=C(n)$ be its $n$-th symmetric product.
 The variety $X$  is smooth and the Albanese map  is  the addition  map $C(n)\to  J(C)$, induced by the Abel-Jacobi map $C\to J(C)$.
 The explicit computation  of the slope  in this case is a bit messy, so we only give here the following estimate, that holds  for $n\le g-2$:
 $$ n!2^n \frac{(g-n-1)^n}{(g-1)(g-2)\dots (g-n)}\le \lambda(X)\le n!2^n\frac{g}{g-n}\le n!2^{n-1}(n+2).$$
Hence for  fixed $n$  the slope of $C(n)$ tends to    $n!2^n$  for $g\to +\infty$.
\end{ex}
\begin{qst}\label{question: (n+1)!} Among the previous examples, only  Examples \ref{ex:complete-int} and \ref{ex:curves} can have birational Albanese map. In both cases $\lambda(X)\ge (n+1)!$. So one may ask whether the inequality $\lambda(X)\ge \frac 5 2 n!$ given in Theorem \ref{teo: Clifford-Severi2} may be strengthened to $\lambda(X)\ge (n+1)!$.

This is an interesting geographical problem already  in the case of surfaces. By \cite[Theorem D]{Ca}, we have that $\lambda(X)\geq 6$ provided $\Omega^1_X$ is globally generated outside a finite number of points. For example, this gives a positive answer in the dimension 2 case, if  $a$ is an immersion. No example  of irregular surface with $\deg a=1$ and slope less than 6 is known.

\end{qst}

     \end{document}